\documentclass{AIMS}
\usepackage{amsmath}
\usepackage{paralist}
\usepackage{graphics}
\usepackage{epsfig}
\usepackage{graphicx}
\usepackage{epstopdf}
\usepackage[colorlinks=true]{hyperref}
\hypersetup{urlcolor=blue, citecolor=red}
\usepackage{hyperref}

\def\currentvolume{X}
 \def\currentissue{X}
  \def\currentyear{2017}
   \def\currentmonth{07}
    \def\ppages{1--22}
     \def\DOI{10.3934/xx.xx.xx.xx}
\usepackage{fullpage}

\usepackage{lineno,hyperref}
\modulolinenumbers[5]
\usepackage{amssymb}
\usepackage{color}
\usepackage{amsmath}

\newtheorem{theorem}{Theorem}
\newtheorem{lem}{Lemma}
\newtheorem{proposition}{Proposition}
\newtheorem{cor}{Corollary}

\newtheorem{conjecture}{Conjecture}
\newtheorem{claim}{Claim}

\newcommand{\R}{{\mathbb R}}

\newcommand{\N}{{\mathbb N}}

\title[A cyclic System with Delay]{A Cyclic System with Delay and its Characteristic Equation}

\email{maelena@ucalgary.ca}
\email{karel.hasik@math.slu.cz}
\email{afi1@psu.edu}
\email{trofimch@inst-mat.utalca.cl}

\thanks{S.\,I.\, Trofimchuk is the corresponding author, e-mail: trofimch@inst-mat.utalca.cl }

\begin{document}
\maketitle

\centerline{\scshape Elena Braverman}
\medskip
{\footnotesize
Dept. of Math. and Stats., University of Calgary,
2500 University Drive N.W., Calgary, AB, Canada T2N 1N4}
\medskip

\centerline{\scshape Karel Hasik}
\medskip
{\footnotesize  Mathematical Institute, Silesian University, 746 01 Opava, Czech Republic}
\medskip

\centerline{\scshape Anatoli F. Ivanov}
\medskip
{\footnotesize Department of Mathematics, Pennsylvania State University, P.O. Box PSU, Lehman PA 18627, USA}

\medskip
\centerline{\scshape Sergei I. Trofimchuk}
\medskip
{\footnotesize  Instituto de Matematica y Fisica, Universidad de Talca, Casilla 747, Talca, Chile}

\medskip

\bigskip

\begin{abstract}
A nonlinear cyclic system with delay and the overall negative feedback is considered.
The characteristic equation of the linearized system is studied in detail. Sufficient conditions for the oscillation of all solutions and for the existence of monotone solutions are derived in terms of roots of the characteristic equation.
\end{abstract}

\section{Introduction}
\label{intro}
This paper studies several aspects of a cyclic system of delay differential equations of the form
\begin{eqnarray}\label{dds}
x_1^{\prime}(t) + \lambda_1x_1(t) & = & f_1(x_2(t-\tau_2))\nonumber \\
x_2^{\prime}(t) + \lambda_2 x_2(t) & = & f_2(x_3(t-\tau_3))\nonumber \\
\dots & \dots & \dots \\
x_{n-1}^{\prime}(t) + \lambda_{n-1} x_{n-1}(t) & = & f_{n-1}(x_n(t-\tau_n))\nonumber \\
x_{n}^{\prime}(t) + \lambda_{n} x_{n}(t) & = & f_{n}(x_1(t-\tau_1)),\nonumber
\end{eqnarray}
where $\lambda_k>0$ and $f_k$ are continuous real-valued functions, $f_k\in C(\mathbb R, \mathbb R),\; 1\le k\le n$.

System (\ref{dds}) has numerous applications in modeling various real world phenomena. Just to mention a few, it was proposed as a mathematical model of protein synthesis processes where natural physiological delays are taken into account \cite{Goo65,Mah80}.
Its scalar ($n=1$) version was used as a mathematical model of multiple processes in physiology, medicine, and physics among others \cite{Had79,HadTom77,MacGla77,M-PNus89,WazLas76}.
Its two-dimensional case $n=2$ was proposed as a model of intracellular circadian rhythm generator \cite{SchKliPenvPe99}. For other applications such as models of neural networks see e.g. \cite{Hop82,Wu01} and further references therein. In the one-dimensional case the corresponding scalar equation
$x^{\prime}(t)+\lambda x(t)=f(x(t-\tau))$
 was comprehensively studied in numerous publications, many of which are summarized as parts of several monographs, see e.g. \cite{DieSvGSVLWal95,Ern09,Hale,Kua93,Smi11}. These monographs also offer an extensive list of references to other applications.
The two-dimensional case was studied in detail in paper \cite{adH79a}, in the form of a second order differential delay equation.
The three-dimensional case of system (\ref{dds}) was analyzed in \cite{IvaBLW04}.

We assume throughout the paper that all functions involved are continuous on $\mathbb R$ and satisfy either positive
\begin{equation}\label{pf}
x\cdot f_k(x)>0,\quad \forall x\ne0,
\end{equation}
or negative feedback
\begin{equation}\label{nf}
x\cdot f_k(x)<0,\quad \forall x\ne0
\end{equation}
hypothesis, and that the overall feedback is negative
\begin{equation*}\label{totalnf}
\prod_{k=1}^n \left(\frac{f_k(x)}{x}\right) <0, \quad \forall x\ne0.
\end{equation*}
In addition, we require that at least one of the functions  $f_k, 1\le k\le n,$ is bounded from one side
\begin{equation}\label{osbd}
f_k(x)\le M\; \text{or}\; f_k(x)\ge-M\;\text{for some}\; M>0 \; \text{and all}\; x\in\mathbb R.
\end{equation}
Besides, we shall assume that each $f_k$ is continuously differentiable in a neighborhood of $x=0$ with $f_k^{\prime}(0)\ne0$.
The sign assumptions (\ref{pf}) and (\ref{nf}) on the nonlinearities $f_k$ imply that system (\ref{dds}) has only one constant solution $(x_1,x_2,\dots,x_n)=(0,0,\dots,0)$.

By a simple change of variables $x_k$ and a modification of functions $f_k$ system (\ref{dds}) can be reduced to a standard form
\begin{equation}\label{sdds}
\begin{array}{lll}
x_1^{\prime}(t) + \lambda_1x_1(t) & = & f_1(x_2(t)) \\
x_2^{\prime}(t) + \lambda_2 x_2(t) & = & f_2(x_3(t)) \\
\dots & \dots & \dots \\
x_{n-1}^{\prime}(t) + \lambda_{n-1} x_{n-1}(t) & = & f_{n-1}(x_n(t)) \\
x_{n}^{\prime}(t) + \lambda_{n} x_{n}(t) & = & f_{n}(x_1(t-\tau)),
\end{array}
\end{equation}
where $\tau=\tau_1+\tau_2+\dots+\tau_n>0$, all $f_k, 1\le k\le n-1$ satisfy positive feedback assumption (\ref{pf}), while $f_n$ satisfies negative feedback assumption (\ref{nf}) \cite{adH79a,IvaBLW04,Mah80}.
Indeed, the new variables $z_k(t), 1\le k\le n,$ can be defined by $z_1(t)=x_1(t), z_2(t)=x_2(t-\tau_2), z_3(t)=x_3(t-(\tau_3+\tau_2)),\dots, z_n(t)=x_n(t-(\tau_n+\dots+\tau_2))$. If a nonlinearity $f_k(x)$ satisfies the negative feedback condition (\ref{nf}) then the new nonlinearity $\hat f_k(x)=f_k(-x)$ will satisfy the positive feedback condition (\ref{pf}). The new variable $\hat z_k=-z_k$ should also be introduced simultaneously.

Our considerations and results throughout the remainder of this paper will refer to the above system (\ref{sdds}) with the specified sign assumptions on functions $f_k, 1\le k\le n$.  The phase space of system (\ref{sdds}) is given by $\mathbb X=C([-\tau,0],\mathbb R)\times\mathbb R^{n-1}$, where  $C([-\tau,0],\mathbb R)$ is Banach space of continuous real valued functions defined on the interval $[-\tau,0]$ with the supremum norm,  and $\mathbb R^{n-1}$ is $(n-1)$ dimensional Euclidean space.
It is easy to see that for arbitrary initial data $(\varphi, x_2,x_3, \dots, x_n)\in \mathbb X$ there exists a uniquely determined corresponding solution
$(x_1(t),x_2(t),\dots, x_n(t))$ to system (\ref{sdds}) defined for all $t\ge0$. Such solution can be found (constructed) by a standard step-cyclic method of integrating system (\ref{sdds}) starting with its last equation.

Without loss of generality, and in order to be specific in assumption (\ref{osbd}), we also assume that the last nonlinearity $f_n$ in system (\ref{sdds}) is bounded from above
\begin{equation}\label{osb}
f_n(x)\le M\; \text{for some}\; M>0 \;\text{and all}\; x\in\mathbb R.
\end{equation}
The possibility when another nonlinearity is one-sided bounded instead can be reduced to this one. The case when $f_n$ is bounded from below is treated along the same reasoning as the case when (\ref{osb}) holds.

The following higher order delay differential equation
\begin{equation}\label{hodde}
\left(
\frac{d}{dt}
+\lambda_1 \right) \left( \frac{d}{dt}+\lambda_2 \right) \cdots
\left( \frac{d}{dt}+\lambda_n \right) x(t) = f(x(t-\tau))
\end{equation}
is closely related to system (\ref{sdds}).
In fact, equation (\ref{hodde}) is a special case of system  (\ref{sdds}) when an appropriate substitution is used to reduce it to a first order system.
Equation (\ref{hodde}) also appears in several applications; a partial case of it was studied in \cite{HalIva93}.

When system (\ref{sdds}) is linearized about the equilibrium $(x_1,x_2,\dots,x_n)=(0,0,\dots,0)$, it results in the following  linear system
\begin{eqnarray}\label{ldds}
x_1^{\prime}(t) + \lambda_1x_1(t) & = & A_1\,x_2(t)\nonumber \\
x_2^{\prime}(t) + \lambda_2 x_2(t) & = & A_2\,x_3(t)\nonumber \\
\dots & \dots & \dots \\
x_{n-1}^{\prime}(t) + \lambda_{n-1} x_{n-1}(t) & = & A_{n-1}\,x_n(t)\nonumber \\
x_{n}^{\prime}(t) + \lambda_{n} x_{n}(t) & = & A_{n}\,x_1(t-\tau),\nonumber
\end{eqnarray}
where $A_k=f_k^{\prime}(0),1\le k\le n.$
The transcendental equation
\begin{equation}\label{ChE}
(z+\lambda_1)(z+\lambda_2) \dots (z+\lambda_n) + ae^{-\tau z}=0,
\end{equation}
where $a=-A_1\cdot A_2\cdot\ldots\cdot A_n>0$, appears then as the characteristic equation of this linear system. The location of zeros of equation (\ref{ChE}) in the complex plane largely determines the behavior of solutions of linear system (\ref{ldds}). In particular, the stability or instability of the trivial solution $(x_1,x_2,\dots,x_n)=(0,0,\dots,0)$ and the oscillation of all solutions are decided by their location.
The primary goal of the present paper is to relate the knowledge about the location of zeros of the characteristic equation (\ref{ChE}) to properties of solutions of the nonlinear system (\ref{dds}).

The structure of the paper is as follows. In Section \ref{roots} we study in detail the characteristic equation (\ref{ChE}), in particular in terms of location of its zeros in the complex plane. Note that in cases $n=1$ or $n=2$, equation (\ref{ChE}) has been comprehensively studied in several papers, see e.g. \cite{adH79a,M-P88}; the case $n=3$ was considered in \cite{IvaBLW04}. Therefore, we emphasize the case $n\ge4$ in this work, comparing the new information with the known facts in the low-dimensional cases $n=1,2,3$. In Section \ref{osc} we establish conditions when all solutions of system (\ref{sdds}) oscillate. In particular, we show that when equation (\ref{ChE}) has no real zeros then all solutions of system ({\ref{sdds}}) oscillate; the converse is also true. Section \ref{nosc} relates real zeros of the characteristic equation (\ref{ChE}) and non-oscillatory solutions of system (\ref{sdds}). The latter has monotone solutions approaching its zero solution if and only if the characteristic equation (\ref{ChE}) has a negative root.
All results for system (\ref{sdds}) in this paper are derived under the standing assumption that functions $f_k$ are continuous with $f_k, 1\le k\le n-1$ satisfying the positive feedback assumption (\ref{pf}), $f_n$ satisfying the negative feedback assumption (\ref{nf}) and being bounded from above in the sense of (\ref{osb}).

\section{The Characteristic Equation}
\label{roots}

This section deals with some of the properties of transcendental equation (\ref{ChE}) and its zeros. We shall treat the real value $a$ as a parameter within the range $\mathbb R_+=\{a\in\mathbb R\;\vert\; a\ge0\}$, while the other constants $\lambda_k>0, 1\le k\le n$ and $\tau>0$ are assumed to be arbitrary but fixed.
\begin{lem}
\label{lem_real}
For system \eqref{sdds} with positive $\lambda_1, \dots, \lambda_n$, and $\tau>0$ exactly one of the following two options is possible:
\begin{enumerate}
\item[(i)]
for any $a>0$, there are no real solutions of its characteristic equation  \eqref{ChE};
\item[(ii)]
there exists $a_0>0$ such that for any $a$ satisfying $0<a\le a_0$, equation \eqref{ChE} has a negative real solution, while for any $a>a_0$ there are no real solutions.
\end{enumerate}
\end{lem}
\begin{proof}
With the polynomial notation
$P_n(z)= (z+\lambda_1)(z+\lambda_2) \dots (z+\lambda_n),$
equation \eqref{ChE} can be rewritten as
$P_n(z)e^{\tau z} = -a.$
Consider the continuous function
\label{th1eq1}
$H(z) = P_n(z)e^{\tau z}.$
For real $z$ we have $\displaystyle \lim_{z \to -\infty} H(z)=0$ and $H(0)>0$. This function may satisfy $H(z) \geq 0$ for any real $z<0$, then the first scenario is implemented, since in this case $H(z)+a>0$ for any real $z$ and $a>0$. Otherwise, $H(z)$ can be negative, and it attains then a negative minimum value $-a_0<0$ for some negative values of $z$. Then for $a\in (0,a_0]$ there is a real root of \eqref{ChE} while for $a>a_0$ there are no real roots.
\end{proof}

Let us note that the first option when \eqref{ChE} never has a real root is possible, for example, when all $\lambda_k$ have even multiplicities.

\begin{lem} \label{scr}
Each complex root $z = \alpha+\beta\,i$, $\beta\not=0$ of the characteristic equation
$$
(z+\lambda_1)(z+\lambda_2)\cdots(z+\lambda_n) + a e^{-z\tau} =0, \quad\tau, a\not=0, \ \lambda_k \in \R, 1\le k\le n,
$$
is simple.
\end{lem}
\begin{proof}  If $z$ has the multiplicity two or higher then differentiating (\ref{ChE}) yields
$$
\sum_{k=1}^n\prod_{j\not=k}(z+\lambda_j) =  a\tau e^{-z\tau} =- \tau  \prod_{j=1}^n(z+\lambda_j).
$$
Therefore for $\alpha=\Re z$, $\beta=\Im z$, we have
$$
\sum_{k=1}^n\frac{1}{z+\lambda_k} = -\tau, \ \mbox{  implying } \  0 = \Im \sum_{k=1}^n\frac{1}{\alpha+\beta\,i+\lambda_k} = -\beta \sum_{k=1}^n\frac{1}{(\alpha+\lambda_k)^2+\beta^2} \not= 0,
$$
which is a contradiction.
\end{proof}

\begin{lem} \label{pure imagi}
Given positive $\lambda_j, 1\le j\le n,$ and $\tau>0$, there exists a uniquely defined strictly increasing sequence $a_k, k\in\mathbb N$, with $\lim_{k\to\infty}a_k=\infty$ and such that the following statements hold:
\begin{itemize}
\item[(i)] For each value $a=a_k$, equation (\ref{ChE}) has exactly one pair of purely imaginary roots $z=\pm\omega_k i$. Besides, the sequence $\omega_k$ is strictly increasing, $0<\omega_1<\omega_2<\omega_3< \dots$, with $\lim_{k\to\infty}\omega_k=\infty$.
\item[(ii)] For all values of $a\in(a_k,a_{k+1})$, equation (\ref{ChE}) has exactly $k$ pairs of complex conjugate roots $z=\alpha_j\pm\beta_j i, 1\le j\le k$ with positive real parts $\alpha_j>0$.
\item[(iii)]
For each complex root $z=\alpha_j+\beta_j i,\, 1\le j\le k$ from part (ii) above its positive real part $\alpha_j$ and its imaginary part $\beta_j$ are strictly increasing functions of $a$ for all $a>0$.  Moreover, $\lim_{a\to\infty}\alpha_j(a)=+\infty$ and $\lim_{a\to\infty}\beta_j(a)=\pi\,(2j-1)/\tau$, with
$\alpha_j(a)>\alpha_l(a)$ and $\beta_j(a)<\beta_l(a)$ for any pair $j,l$ such that $1\le j<l\le k$ and all $a>0$.
\end{itemize}
\end{lem}
\begin{proof} (i)\,
Assume that equation (\ref{ChE}) has a pair of purely imaginary roots $z=\pm\omega i$, $\omega>0.$ Substituting $z=\omega\, i$ into characteristic equation (\ref{ChE}) yields
$$
\prod_{j=1}^n (\lambda_j+\omega\, i)+a\exp\{-\omega\tau i\}=0.
$$
By representing each value $\lambda_j+\omega i, 1\le j\le n,$ as
$$
\lambda_j+\omega i=\rho_j\,\exp\{i\,\theta_j\}, \quad \text{where}\quad \rho_j(\omega)=\sqrt{\omega^2+\lambda_j^2} ,\; \theta_j(\omega)=\tan^{-1}\left(\frac{\omega}{\lambda_j}\right),
$$
we arrive at the following equation
$$
\rho_1\,\rho_2\,\dots\rho_n\, \exp\left\{\left(\sum_{j=1}^n\theta_j\right)i\,\right\}=a\, \exp\{-i\,\omega\tau+i\,\pi(2k-1)\},\; k\in \mathbb N,
$$
which is equivalent to the system
\begin{equation}\label{polar}
\prod_{j=1}^n(\omega^2+\lambda_j^2)=a^2,\quad \sum_{j=1}^n\theta_j=-\omega\tau+\pi (2k-1),\; k\in\mathbb N.
\end{equation}

For each $j\in\{1,2,\dots,n\}$ the corresponding component $\theta_j(\omega)$ is an increasing and concave down function in $\omega>0$ with $\theta_j(0)=0$ and $\theta_j(+\infty)=\pi/2$. This in turn implies that the function $\Theta_0(\omega)=\sum_{j=1}^n\theta_j$ is increasing and concave down for $\omega\ge0$ with $\Theta_0(0)=0$ and $\Theta_0(+\infty)=n\pi/2$.
For the family of lines $y=-\omega\tau+\pi(2k-1)$, $k\in\mathbb N,$ there exists exactly one intersection of each of the lines
with the graph of $y=\Theta_0(\omega)$. The corresponding positive values $\omega_k$, $k \in \mathbb N,$ are strictly increasing, $0<\omega_1<\omega_2<\dots$, with $\omega_k\to+\infty$ as $k\to\infty$,
see Fig.~\ref{figure1}.
\begin{figure}[ht]
\includegraphics[scale=0.42]{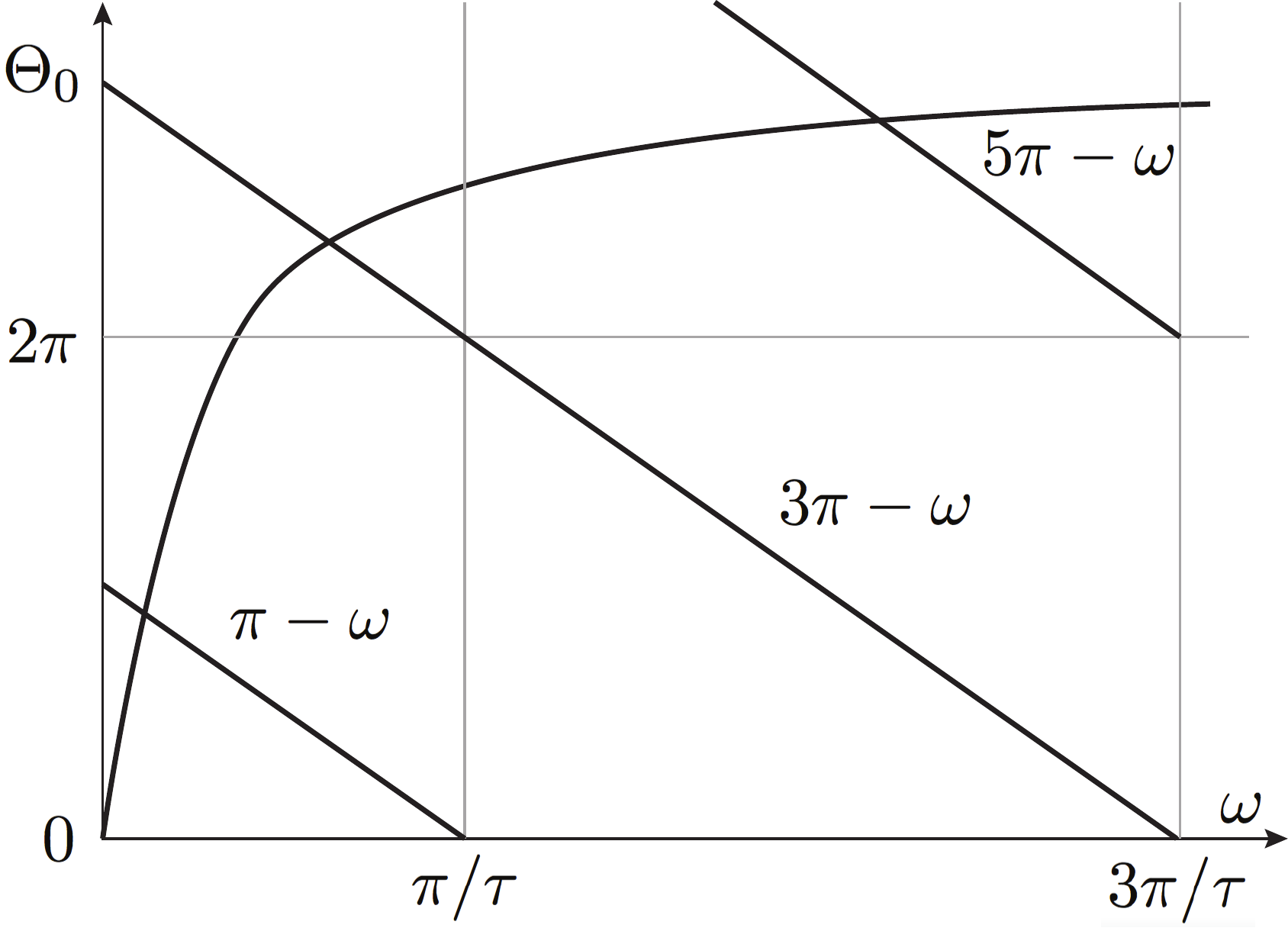}\hspace{3mm}
\includegraphics[scale=0.42]{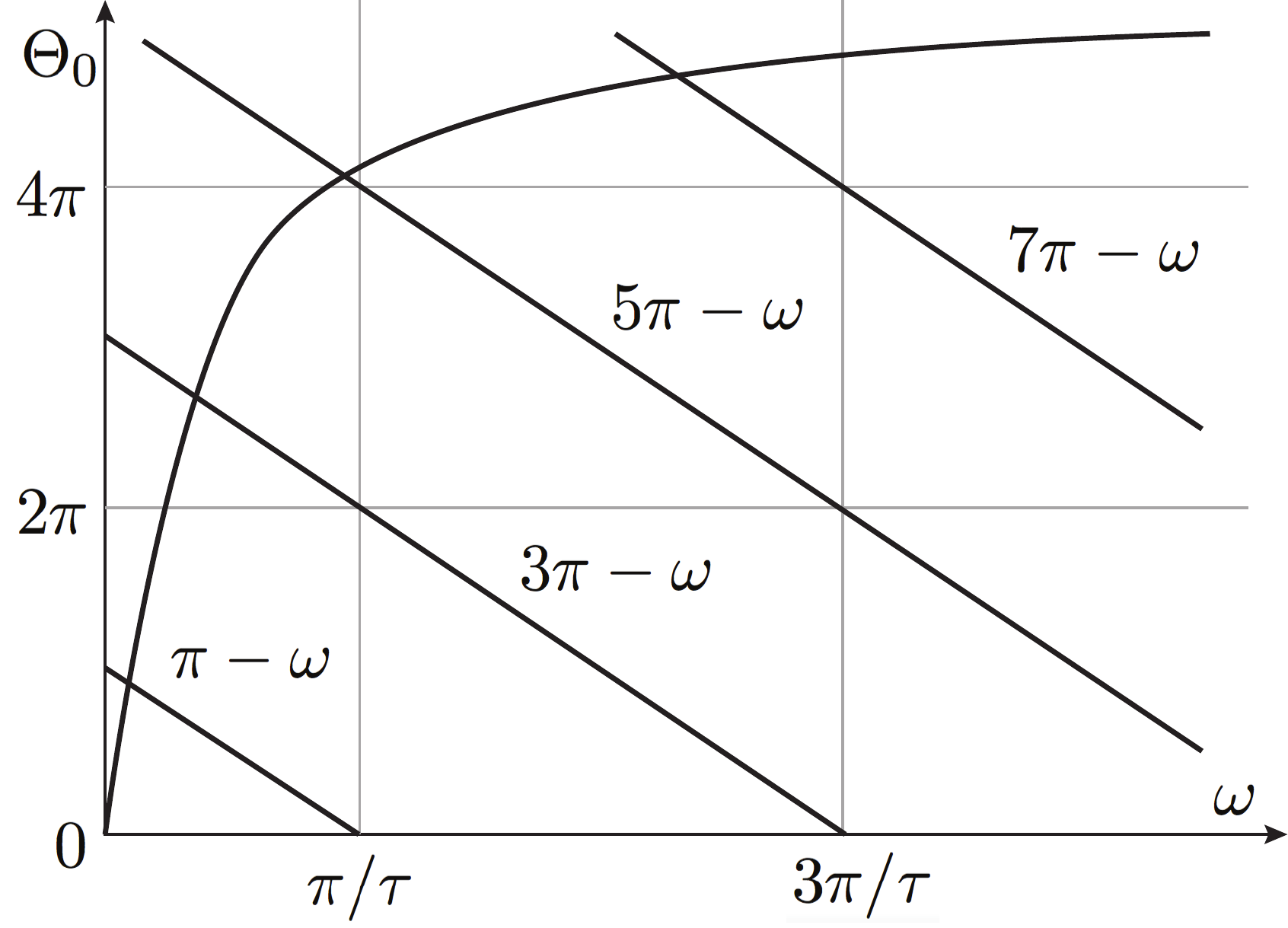}
\caption{The graphs of $\Theta_0(\omega)=\sum_{j=1}^n\theta_j$ and $y=-\omega\tau+\pi(2k-1)$, $k\in\mathbb N$ for
$n=6$ (left) and $n=10$ (right).}
\label{figure1}
\end{figure}
One can also have an asymptotic representation for $\omega_k$ when $k\to\infty$ as $\omega_k\sim (\pi/\tau)(2k-1-\frac{n}{2})$.
Note that the values of $\omega_k$ are independent of the parameter $a$ as they are determined by the second equation of (\ref{polar}).

Given a particular value $\omega_k$ one can use then the first equation of system (\ref{polar}) to determine the corresponding value of $a_k$ as
$$
a_k=\sqrt{\prod_{j=1}^n(\omega_k^2+\lambda_j^2)} \, , \quad k\in \mathbb N.
$$
It is straightforward to see that $0<a_1<a_2<a_3<\dots$ with $a_k\to+\infty$ as $k\to\infty$.
This proves part (i) of the lemma.

(ii)\, By differentiating equation (\ref{ChE}) with respect to the parameter $a$ we obtain
\begin{equation*}
P_n^\prime(z)\, z^\prime+\exp\{-z\tau\}-a\tau\exp\{-z\tau\}\,z^\prime=0\;   \Longrightarrow  \; z^\prime\left[P_n^\prime(z)+\tau\, P_n(z)\right]=\frac{P_n(z)}{a}.
\end{equation*}
Clearly $z^\prime\ne0$, so one has the representation
\begin{equation}\label{z'}
\frac{1}{az^\prime}=\frac{P_n^\prime(z)}{P_n(z)}+\tau=\sum_{j=1}^{n}\frac{1}{z+\lambda_j}+\tau.
\end{equation}
By taking the real part of equation (\ref{z'}) and evaluating it at $z=i \omega_k$, we have
$$
\Re\left(\frac{1}{az^\prime(a)}\right)=\tau+\sum_{j=1}^{n}\Re\left(\frac{1}{\lambda_j+i\omega_k}\right)=
\tau+\sum_{j=1}^{n}\left(\frac{\lambda_j}{\lambda_j^2+\omega_k^2}\right)>0
\Longrightarrow\; \Re\, z^\prime(a_k)>0, \; \forall k\in \mathbb N .
$$
Therefore, a pair of complex conjugate solutions $z(a)=z_k(a)=\alpha_k(a)\pm i\,\beta_k(a), k\in\mathbb N,$ to equation (\ref{ChE}) crosses the imaginary axis at the parameter value $a=a_k$ with
$$
\alpha_k(a_k)=0,\; \beta_k(a_k)=\omega_k\quad\text{and}\quad \alpha_k(a)>0\;\;\forall\; a>a_k.
$$
By the same reason, no zero of the characteristic function can leave the right half-plane as $a>0$ increases. This means that the roots of equation (\ref{ChE}), after crossing the imaginary axis, remain in the right half-plane as $a$ continues to increase. They also cannot intersect the interval $[0,+\infty)$ as (\ref{ChE}) has no positive solutions.
All this proves part (ii) of the lemma.

(iii) Substituting  the derivative $z^\prime(a)=\alpha^\prime(a)+\beta^\prime(a)\, i$ of a zero $z=\alpha+\beta\, i$ into equation (\ref{z'}) one arrives at the identity
$$
\frac{\alpha^\prime-i\beta^\prime}{a[(\alpha^\prime)^2+(\beta^\prime)^2]}=\tau+\sum_{j=1}^{n}\frac{\alpha+\lambda_j-i\beta}{(\alpha+\lambda_j)^2+\beta^2}.
$$
By separating the real and the imaginary parts one gets the inequalities
$$
\alpha^\prime(a)=a \left[ (\alpha^\prime)^2+(\beta^\prime)^2 \right]
\left[ \tau+\sum_{j=1}^{n}\frac{\alpha+\lambda_j}{(\alpha+\lambda_j)^2+\beta^2} \right] >0,\quad\forall a>0,
$$
and
\begin{equation}\label{up}
\beta^\prime(a)=\beta a \left[ (\alpha^\prime)^2+(\beta^\prime)^2\right]
\left[ \sum_{j=1}^{n}\frac{1}{(\alpha+\lambda_j)^2+\beta^2} \right]>0,\quad \forall a>0,
\end{equation}
which proves the monotonicity.

With $z=\alpha+\beta\,i$, the characteristic equation (\ref{ChE})
$$
(\alpha+\beta\,i+\lambda_1)(\alpha+\beta\,i+\lambda_2)\dots(\alpha+\beta\,i+\lambda_n)+a\exp\{-\tau(\alpha+\beta\,i)\}=0
$$
can be rewritten in the polar coordinates form
\begin{equation*}
\rho_1\,\rho_2\,\ldots\,\rho_n\, \exp\left\{\left({\text{Arg}}(\alpha+\beta\,i+\lambda_1)+\dots+{\text{Arg}}(\alpha+\beta\,i+\lambda_n)\right)\,i\right\}=
a\exp\{-\alpha\tau\}\exp\{(\pi-\beta\tau)\,i\},
\end{equation*}
or
\begin{equation}\label{ChE-P}
\sqrt{\prod_{j=1}^n\left[(\alpha+\lambda_j)^2+\beta^2\right]}\,
\exp\left\{\left(\sum_{j=1}^n\tan^{-1}\left(\frac{\beta}{\alpha+\lambda_j}\right)\right)\,i\right\}=a\exp\{-\alpha\tau\}\exp\{(\pi-\beta\tau)\,i\},
\end{equation}
where
$$
\rho_j=\sqrt{(\alpha+\lambda_j)^2+\beta^2},\qquad \text{Arg}\,(\alpha+\beta\,i+\lambda_j)=\tan^{-1}\left(\frac{\beta}{\alpha+\lambda_j}\right),\quad 1\le j\le n.
$$
Equation (\ref{ChE-P}) is equivalent to the system
\begin{equation}\label{ChE-s}
\sqrt{\prod_{j=1}^n\left[(\alpha+\lambda_j)^2+\beta^2\right]}\,\exp\{\alpha\tau\}=a,\qquad
\sum_{j=1}^n\tan^{-1}\left(\frac{\beta}{\alpha+\lambda_j}\right)=(2m-1)\pi-\beta\tau,\; m\in\mathbb N.
\end{equation}
The first equation of system (\ref{ChE-s}) implies the following
\begin{claim}\label{prop0}
Given an arbitrary $\alpha_0\in\mathbb R$, every vertical half-line $z=\alpha_0+\beta\,i, 0\leq \beta<\infty,$ in the complex plane $\mathbb C$ contains at most one solution of the characteristic equation (\ref{ChE}). For any fixed $\beta_0$, every horizontal half-line $z=\alpha+i\,\beta_0, 0<\alpha<\infty,$ contains at most one solution of (\ref{ChE}).
\end{claim}
This follows from the monotonicity in $\alpha$ and $\beta$ of the left-hand side of the first equation of (\ref{ChE-s}).

Consider now the first pair of complex conjugate solutions $z(a)=\alpha(a)\pm \beta(a)\,i$ to (\ref{ChE}) which appears from a pair of purely imaginary solutions $\pm\omega_1\,i$ at $a=a_1$ (see part (i) of this lemma for additional details). Due to the continuous dependence of  $\alpha(a),  \beta(a)$ on $a$, this pair solves the second equation of (\ref{ChE-s}) with  fixed $m=1$ for all $a>0$. Now,
for every $\alpha>0$ the function $\Theta_\alpha(\beta)=\sum_{j=1}^n\tan^{-1}\left(\frac{\beta}{\alpha+\lambda_j}\right)$ is concave down and increasing for all $\beta>0$. Therefore, there is precisely one intersection of the curve $Y=\Theta_\alpha(\beta)$ with the line $Y=\pi-\beta\tau$ on the interval $[0,\pi/\tau]$. Because of this uniqueness,  the abscissa  of the intersection point  should coincide  with $\beta(a)$. Clearly,  $\omega_1<\beta<\pi/\tau$. Since $\beta$ is increasing in $a>0$, the limit $\lim_{a\to\infty}\beta(a)=\beta^*\le\pi/\tau$ exists. The first equation of system (\ref{ChE-s}) then immediately implies $\lim_{a\to\infty}\alpha(a)=+\infty.$ If one assumes next that $\beta^*<\pi/\tau$ then taking the limit in the second equation of (\ref{ChE-s}) as $a\to\infty$ leads to the contradiction $0=\pi-\tau\,\beta^*>0.$ This proves the limits in part (iii) of the lemma when $m=1$. The other cases when $m\ge2$ are similar and left to the reader.

The last assertions  of part (iii) that $\alpha_j(a)>\alpha_l(a)$ and $\beta_j(a)<\beta_l(a)$ for $1\le j<l\le k$ follow from Claim \ref{prop0} together with the continuous dependence of $\alpha_j(a), \beta_j(a)$ on $a>0$. Note that this also means that the pair of complex conjugate solutions $z=\alpha\pm\beta\,i$ that appears from the pair $z=\pm\omega_1\,i$ at $a=a_1$ is the leading one in the sense that its real part $\alpha>0$ is the largest among all other solutions of the characteristic equation (\ref{ChE}) for any $a\ge a_1$. Its imaginary part always satisfies $0<\beta<\pi/\tau,\; \forall a\ge a_1.$ Both $\alpha$ and $\beta$ are increasing functions in $a$ for $a\ge0$.
\end{proof}

\begin{cor}\label{corol1a} Suppose that $a_0$ from Lemma 1 exists, then
$a_0 < a_{[(n+1)/2]}.$
\end{cor}
\begin{proof}  First note that Lemma \ref{scr} and inequality (\ref{up}) imply that each complex root $z_j(a)$ of equation  (\ref{ChE})  in the upper half-plane  depends smoothly on $a$ and its imaginary part $\Im z_j(a)>0$ is a strictly increasing function of $a$.  Hence, when $a$ is increasing,  complex conjugate roots can not meet at a point of the real axis. In addition, Claim 1 assures that the relative location
of the real part of a pair of  complex conjugate roots and the real part of any other root of equation  (\ref{ChE})   does not change when $a>0$ is increasing.

Next, for $a=0$ equation  (\ref{ChE}) has $n$ real roots $z_j(0) = -\lambda_j, \ j =1,\dots, n$. Set $\lambda_\circ :=   \max\{\lambda_j, \ j =1,\dots, n\}$. Thus,  by Rouch\'e theorem \cite{Con78}, equation  (\ref{ChE})  has exactly $n$ roots in the half-plane $\{z: \Re z \geq - \lambda_\circ -1\}$,  for all sufficiently small $a\geq 0$.  When $a>0$ increases, the roots $z_j(a), \ j =1,\dots,n,$  change continuously
(possibly, some of them taking complex values)  and, as we have already mentioned, there exists $a' >0$ such that $\Re z_j(a)<\min\{\Re z_i(a), \ i =1,\dots,n-1\}$ for all $j >n$ and $a \in (0,a']$.

Our next arguments are dependent on the parity of $n$. We begin with the simpler case of even $n$.
When $n=2m$ the polynomial $p(z)=(z+\lambda_1)\dots (z+\lambda_n)$ is positive for all $z \in (-\infty,  \lambda_\circ)$  and therefore roots $z_j(a), \ j >n,$ are complex for all small $a>0$.
Consequently, if $n$ is even, the roots $z_j(a), \ j >n,$ remain  complex  and the following holds as $a>0$ increases:
 (i) the relative location of their real parts $\Re z_j(a), \ j >n,$ does not change; 
 (ii) $\Re z_j(a) < \min\{\Re z_i(a), \ i =1, \dots, n\}$ for all $j >n$ and $a >0$.  In particular, the roots $z_j(a), \ j >n=2m,$  can cross the imaginary axis only when the roots $z_j(a), \ j =1, \dots, 2m,$ have already left the left half-plane.

When  $n=2m+1$ is odd $\lim_{z\to -\infty}(z+\lambda_1)\dots (z+\lambda_n) =-\infty$, so equation  (\ref{ChE})  has at least one negative root
for all sufficiently small $a>0$:
$$
z_-(a) \ll \min\{\Re z_i(a), \ i =1, \dots, n\}.
$$
In particular, this shows that the parameter value $a_0>0$ from Lemma \ref{lem_real} always exists when $n$ is odd.   We claim that $z_-(a)$  is  actually a
unique negative root  in the interval $(-\infty, -\lambda_\circ-1)$ for all small $a>0$
(in particular, it is simple). Indeed, observe first that  $w(a): = |z_-(a)|= -z_-(a)$ satisfies the relation
\begin{equation}\label{peq}
r(w,a):= w\tau + \ln a = \sigma(w)=:\sum_{j=1}^n\ln(w-\lambda_j),
\end{equation}
where function $\sigma(w)$ is defined and strictly concave down for all $w > \lambda_\circ$.
The later property of $\sigma(w)$ implies that equation (\ref{peq}) can have at most two solutions.
Since $r(\lambda_\circ +1,a) < \sigma(\lambda_\circ +1)$ for all small $a$ and, for each fixed $a>0$, there exists $w_1> \lambda_\circ +1$ such that $r(w_1,a) > \sigma(w_1)$,  we can conclude that $w(a)$ is actually a unique positive root
of  (\ref{peq}) belonging to the interval $(\lambda_\circ +1, +\infty)$. This proves the claim.

Now, if we denote the root $z_-(a)$ as $z_{n+1}(a)$, then all other roots $z_j(a)$, $j > n+1$, of (\ref{ChE}) should be complex conjugate pairwise, and the relative location of their real parts should be preserved for all $a >0$ (see the first paragraph of the proof). This implies that, as $a>0$ increases,  the negative root $z_{n+1}(a)$ has to merge with the most negative real root among $\{z_1(a),\dots,z_n(a)\}$
 at some value $\tilde a >0$ giving rise to a new pair of complex conjugate roots. Using the same notation $z_{n+1}(a)$ for one of them, we conclude that
 $\Re z_{n+1}(a) > \Re z_j(a)$ for all $a >0$ and $j >n+1$. In particular, the roots $z_j(a), \ j >n+1,$  can cross the imaginary axis only after  the roots $z_j(a), \ j =1, \dots, 2m+2$ have already crossed it (as complex conjugate pairs).

The above discussion shows that independently of the parity of $n \in \N$,  the left half-plane does not contain real roots of equation  (\ref{ChE})  for all $a \geq a_{[(n+1)/2]}$.
\end{proof}

\section{Oscillation}\label{osc}

In this section we relate certain properties of the characteristic equation (\ref{ChE}) with the oscillatory properties of nonlinear system (\ref{sdds}).
In what follows,  $|\cdot|$ denotes the Euclidean norm in  the vector space $\R^m$ (we use the same symbol for all dimensions $m$). The corresponding operator norm of a linear map: $A:\R^m\to \R^m$ is denoted by $\|A\|$.

Recall that a continuous function $u(t), t\ge0$ is called oscillatory (about $0$) if there is a sequence $t_n\to+\infty$ such that $u(t_n)=0$. Usually a function that identically equals to zero for large $t$, $u(t)\equiv0, \forall t\ge T\ge0$, is excluded, and is not viewed as oscillating. If a function does not oscillate in the above sense it is called non-oscillatory.

Since system (\ref{sdds}) has only one equilibrium $(x_1,x_2,\dots,x_n)=(0,0,\dots,0)$, we are interested in solutions that oscillate about it. Due to the sign hypotheses imposed on the nonlinearities $f_k, 1\le k\le n,$ it is easy to see that if one of the components $x_k$ is oscillatory, then all the remaining components are oscillatory as well. Likewise, if one of the components $x_k$ is eventually of a definite sign (say, $x_k(t)>0$ or $x_k(t)<0$ for all $t\ge T>0$ and some $T>0$) then all the remaining components are also of a definite sign. Moreover, it can be easily shown that all components of any non-oscillatory solution $(x_1,\dots,x_n)$ go to zero as $t$ goes to infinity, $\lim_{t\to+\infty} |x_j(t)|=0$ for all $1\le j\le n$. See Lemma \ref{L2} below which provides more precise information about non-oscillating solutions of system (\ref{sdds}).  Thus, excluding eventually trivial solutions, every solution  to system (\ref{sdds}) either oscillates or is eventually of a definite sign with every component decaying to zero in the latter case. The oscillation of solutions to system (\ref{sdds}) happens in a stronger sense than that given by the above definition. It will be seen from the considerations of this section that if a component $x_k$ oscillates about zero then there exists an increasing sequence $s_j\to\infty$ such that $x_k(s_j)\cdot x_k(s_{j+1})<0.$

\begin{lem}
Suppose that $f_j\in C(\mathbb R,\mathbb R), 1\le j\le n$, each function $f_j, 1\le j\le n-1$ satisfies positive feedback assumption (\ref{pf}) while $f_n$ satisfies negative feedback assumption (\ref{nf}) and is one-sided bounded as specified by (\ref{osb}). Then every
solution $\mathbf x(t) = (x_1(t),x_2(t), \dots, x_n(t)), \ t \geq0,$ to system (\ref{sdds}) is  bounded. Moreover,  a constant $K>0$ can be indicated, which depends on $\tau, \lambda_1,\lambda_2,\dots,\lambda_n$ and $f_1,f_2,\dots, f_n$ only,  and such that for the solution $\mathbf x(t)$ there exists a time moment $t_x\ge0$ such that $|\mathbf x(t)|\le K, \; \forall t\ge t_x.$
\end{lem}

We shall prove this lemma in the context of more general considerations of the following system
\begin{eqnarray}\label{sdde}
\varepsilon_1 x_1^{\,\prime}(t) &=& -x_1(t)+F_1(x_2(t))\nonumber \\
\varepsilon_2 x_2^{\,\prime}(t) &=& -x_2(t)+F_2(x_3(t))\nonumber \\
                 \ldots &{\,}& \ldots \\
\varepsilon_{n-1}x_{n-1}^{\,\prime}(t) &=& -x_{n-1}(t)+F_{n-1}(x_n(t))\nonumber\\
\varepsilon_n x_n^{\,\prime}(t) &=& -x_n(t)+F_n(x_1(t-\tau)),\nonumber
\end{eqnarray}
with arbitrary positive parameters $\varepsilon_j, 1\le j\le n,$ and general nonlinearities $F_j$ satisfying the same sign assumptions as $f_j$ in system (\ref{sdds}). System (\ref{sdds}) can be rewritten in the above form with $\varepsilon_j=1/\lambda_j>0$ and $F_j=(1/\lambda_j)\, f_j$, $1\le j\le n.$
In the limit case $\varepsilon_j\to0^+\;$  (\ref{sdde}) results in the system of pure difference equations
$$
x_1(t)=F_1(x_2(t)),\quad x_2(t)=F_2(x_3(t)),\quad \ldots\quad x_{n-1}(t)=F_{n-1}(x_n(t)),\quad x_n(t)=F_n(x_1(t-\tau)),
$$
which is further reduced to a single difference equation
$$
x_1(t)=(F_1\circ F_2\circ\dots\circ F_n)(x_1(t-\tau))=:F(x_1(t-\tau)).
$$
The asymptotic properties of the latter are well known, including the case of continuous time $t\ge0$ \cite{ShaMaiRom93}.
They are essentially determined by dynamical properties of the one-dimensional map behind the continuous function $F$:
\begin{equation*}\label{im}
F(x)=(F_1\circ F_2\circ\dots\circ F_n) (x).
\end{equation*}

We assume now that the one-dimensional map $F$ has an invariant interval $I\subset \mathbb R$, in the sense that the inclusion $F(I)\subseteq I$ holds.
Define next the intervals $I_k, 2\le k\le n,$ inductively by
\begin{equation*}\label{incl}
I_n=F_n(I), I_{n-1}=F_{n-1}(I_n)=(F_{n-1}\circ F_n)(I),\dots, I_2=F_2(I_3)=(F_2\circ\dots\circ F_n)(I).
\end{equation*}
Then $F_1(I_2)\subseteq I$ in view of the invariance of $I$ under $F$. Define next a subset $\mathbb X_I$ of the phase space $\mathbb X=C([-\tau,0],\mathbb R)\times\mathbb R^{n-1}$ of system (\ref{sdde}) by
\begin{eqnarray*}\label{X_I}
\mathbb X_I&=&C([-\tau,0], I)\times I_2\times I_3\times\dots\times I_n=\nonumber\\
       {\;}&=&\{\Psi=(\varphi(s),x_2^0,x_3^0,\dots,x_n^0)\in\mathbb X\;\vert\; \varphi(s)\in I\;\:\:\forall s\in[-\tau,0],x_2^0\in I_2,x_3^0\in I_3,\dots,x_n^0\in I_n \}.
\end{eqnarray*}
In view of the sign assumptions on the nonlinearities $f_k, 1\le k\le n,$ one easily sees that the function $F(x)$ satisfies the negative feedback condition, with $x=0$ being its only fixed point. Therefore, the constant solution $\mathbf x=(0,0,\dots,0)$ is the only equilibrium of system (\ref{sdde}).
Also, $F(x)$ is one-sided bounded in the sense of (\ref{osbd}) if such is at least one of the nonlinearities $F_k, 1\le k\le n$.

The following lemma provides an invariance property for system (\ref{sdde}). It shows that given arbitrary initial function $\Psi\in\mathbb X_I$ the  corresponding solution to system (\ref{sdde}) satisfies $\mathbf x(t,\Psi)\in\mathbb X_I$ for all $t\ge0$.

\begin{lem}\label{invar1}
Suppose that an interval $I$ is invariant under the map $F$ and an initial function $\Psi$ belongs to $\mathbb X_I$. Let $\mathbf x=\mathbf x(t,\Psi)=(x_1,x_2,\dots, x_n)$ be the corresponding solution to system (\ref{sdde}). Then the following inclusions hold
$$
x_1(t)\in I, x_2(t)\in I_2, \dots, x_n(t)\in I_n\quad{\text{for all}}\quad t\ge0
$$
and any positive  values of parameters $\varepsilon_1,\dots,\varepsilon_n$.
\end{lem}

The proof of this lemma is based on a simple fact related to the solution of the initial value problem for the scalar differential equation
\begin{equation}\label{IVP}
\varepsilon u^\prime(t)+u(t)=b(t),\quad u(t_0)=u_0,
\end{equation}
with $\varepsilon>0$ and a continuous function $b(t), t\ge t_0$.

\begin{proposition}\label{prop1}
Suppose that the range of the function $b(t)$, $t\ge t_0,$ is an interval $L$ and $u(t_0)\in L$. Then the solution $u(t)$ to the initial value problem (\ref{IVP}) satisfies $u(t)\in L,\; \forall t\ge t_0.$
\end{proposition}
\begin{proof}
Indeed, assume the claim is not valid, and let  $t_1\ge t_0$ be the first point of exit of the solution $u(t)$ from the interval $L=[c,d]$. To be definite, assume that $u(t_1)=c, u(t)\in[c,d]\;\forall t\in[t_0,t_1]$, and every right neighborhood $(t_1,t_1+\delta)$ contains a point $t^\prime$ such that $u(t^\prime)<c$.
Then it also contains a point $t^{\prime\prime}$ such that $u(t^{\prime\prime})<c$ and $u^\prime(t^{\prime\prime})<0.$ However, equation (\ref{IVP}) then implies that $u^\prime(t^{\prime\prime})=(1/\varepsilon)[b(t^{\prime\prime})-u(t^{\prime\prime})]>0$, a contradiction.
The other possibility $u(t_1)=d, u(t)\in[c,d]\;\forall t\in[t_0,t_1]$ is considered along the same reasoning.
\end{proof}

Lemma \ref{invar1} can be proved now by using the cyclic structure of system (\ref{sdde}).
Given an initial value $\Psi=(\varphi(s),x_2^0,x_3^0,\dots,x_n^0)\in\mathbb X_I$, the last equation of system (\ref{sdde}) implies, in view of Proposition \ref{prop1}, that $x_n(t)\in I_n\; \forall t\in[0,\tau]$. Likewise, going upward along equations of the system  one sees that $x_{n-1}(t)\in I_{n-1},\dots,x_2(t)\in I_2$ and $x_1(t)\in I$ for all $t\in[0,\tau].$ By repeating the reasoning  cyclically again one concludes that
$x_n(t)\in I_n, x_{n-1}(t)\in I_{n-1},\dots,x_2(t)\in I_2, x_1(t)\in I, \forall t\in[\tau,2\,\tau]$. Then by induction these inclusions hold for all $t\ge0$.
\hfill $\square$

We construct now an invariant interval $I_0$ for the map $F$ based on the one-sided boundedness of the function $F_n$. Indeed with assumption (\ref{osb}) on the nonlinearity $f_n(x)$, $f_n(x)\le M$ for all $x\in\mathbb R$ and some $M>0$, one easily sees that the function $F=F_1\circ F_2\circ\dots\circ F_n$ is also one-sided bounded, $F(x)\le M_1$ for all $x\in\mathbb R$  and some $M_1>0$. Besides $F(x)$ satisfies the negative feedback condition (\ref{nf}).
Define next $\alpha=\min\{F(x), x\in [0,M_1]\}$. Then the interval $I_0:=[\alpha,M_1]$ is invariant under the map $F:$ $F(I_0)\subseteq I_0.$
Therefore, any solution $\mathbf x(t,\Psi)$ to system (\ref{sdde}) with the initial function $\Psi\in\mathbb X_{I_0}$ is uniformly bounded,
in view of Lemma \ref{invar1}.

We shall show next that every solution $\mathbf x(t)$ of system (\ref{sdde}), with an arbitrary initial function $\Psi\in\mathbb X$,  enters the set $\mathbb X_{I_0}$ in finite time.

\begin{lem}\label{lemma2}
Given an arbitrary initial function $\Psi\in\mathbb X$, there exists a finite time $T=T(\Psi)\ge0$ such that the corresponding solution $\mathbf x=\mathbf x(t,\Psi)$ to system (\ref{sdde}) satisfies $\mathbf x(t)\in \mathbb X_{I_0}$ for all $t\ge T$.
\end{lem}

In order to prove Lemma \ref{lemma2} we need another simple auxiliary result about solutions to the initial value problem (\ref{IVP}).

\begin{proposition}\label{prop2}
Suppose that the range of $b(t), t\ge t_0,$ is an interval $L$ and $u(t_0)\not\in L$. Then one of the following holds for the solution $u(t)$ to the initial value  problem (\ref{IVP}):
\item[(i)] there exists a finite  $t_1\ge t_0$ such that $u(t)\in L$ for all $t\ge t_1$;
\item[(ii)] $u(t)$ is monotone for all $t\ge t_0$ and $\lim_{t\to\infty}u(t)=l^*$ where $l^*$ is one of the endpoints of the interval $L$.
\end{proposition}
\begin{proof}
To be definite, assume that $u(t_0)>d$ where $L=[c,d].$ Then the solution $u(t)$ is decreasing in some right neighborhood of $t=t_0$. If there exists a finite time $T>0$ such that $u(T)=d$ then, by Proposition \ref{prop1}, the solution satisfies $u(t)\in L$ for all $t\ge T$. If, on the other hand, $u(t)>d$ for all $t\ge t_0$, then $u(t)$ has the limit $\lim_{t\to\infty} u(t)=l^*\ge d.$ If $l^*>d$ then $\dot u(t)=(1/\varepsilon)[-u(t)+b(t)]\le[(d-l^*)/\varepsilon]<0, t\ge t_0,$ leading to the contradiction $u(t)\to-\infty$ as $t\to\infty$. The case $u(t_0)<c$ is treated similarly.
\end{proof}

\begin{cor}\label{corol1}
Suppose that $\lim_{t\to\infty} b(t)=b_0.$ Then the solution to the initial value problem (\ref{IVP}) also satisfies $\lim_{t\to\infty} u(t)=b_0$.
\end{cor}
\begin{proof}
It follows from Proposition \ref{prop2}, as for an arbitrary $\delta>0$, the interval $L$ can be chosen as $[b_0-\delta,b_0+\delta]$.
Also, this claim can be proved directly by using the variation of constant formula for the solution of (\ref{IVP}).
\end{proof}

Now we are in the position to prove Lemma \ref{lemma2}.
Let an initial function $\Psi=(\varphi(s),x_2^0,\ldots,x_n^0)\in\mathbb X$ be given. Suppose that, for some $j \in \{1,\dots,n\}$, $t_{j+1} \in \R, $  and some non-empty interval $A_{j+1}\subseteq \R$ it holds that $x_{j+1}(t) \in A_{j+1}$ for all $t \geq t_{j+1}$ (here we identify $n+1$ with $1$).  We claim that then  there  exists $s_j$ such that
\begin{equation}\label{incl2}
x_j(t)\in F_j(A_{j+1}),  \ t\ge s_j.
\end{equation}
Suppose not, then, using the $j$-th equation of system (\ref{sdde}) and in view of Proposition \ref{prop2}, (ii), one sees that $x_j(t)$ is monotone for all large $t$ with $\lim_{t\to\infty} x_j(t)=x_j^*\ne0.$ By using next $(j-1)$-st equation of system (\ref{sdde}) and Corollary \ref{corol1} one concludes that $\lim_{t\to\infty} x_{j-1}(t)=F_{j-1}(x_j^*)\ne0$ (here we identify $0$ with $n$). Going up along successive equations of system (\ref{sdde}), and calculating $\lim_{t\to\infty} x_k(t)$ for all $k$, one concludes that $x_j^*$ satisfies the iterative equation
$$
x_j^*=F_j\circ F_{j+1}\circ\cdots\circ F_{n}\circ F_1\circ F_2\circ\cdots\circ F_{j-1}(x_j^*):=\hat{F}_j(x_j^*).
$$
Since $\hat{F}_j$ satisfies the negative feedback condition (\ref{nf}) its only fixed point is $x=0$, a contradiction with $x_j^*\ne0$.
Therefore, there exists $s_j\ge0$ such that $x_j(t)\in F_j(A_{j+1})$ for all $t\ge s_j.$

By taking first $A_j:=\R$ for all $j$, we conclude that  there exists $S_1 = \max\{s_j, j =1,\dots,n\}$ such that
$$
x_j(t)\in F_j(\R),  \ \mbox{for all} \ t\ge S_1, \ j =1,\dots,n.
$$
Consequently, if we set $A_j:= F_j(\R)$, we can infer the existence of  $S_2$ such that
$$
x_{j-1}(t)\in F_{j-1}(F_j(\R)) =:A_{j-1},  \ \mbox{for all} \ t\ge S_2, \ j =1,\dots,n.
$$
Applying the same argument repeatedly $n$ times we conclude that there exists $T_1\ge0$ such that the following inclusions hold
\begin{equation}\label{incl3}
x_1(t)\in F(\R), x_{2}(t)\in\hat{F}_{2}(\R),\ldots, x_n(t)\in\hat{F}_{n}(\R)\quad \forall t\ge T_1,
\end{equation}
where $F$ is given by (\ref{im}), and $\hat{F_k}, k=2,3,\dots,n,$ are composite functions defined by
$$
\hat{F_k}:=F_k\circ\dots\circ F_n\circ F_1\circ\dots\circ F_{k-1}.
$$
Furthermore, repeating the same argument  $3n$ times, we find that there exists $T' \ge T_1$ such that
$$
x_1(t)\in F^3(\R) \subset I_0, x_{2}(t)\in\hat{F}_{2}^3(\R)  \subset (F_2\circ\dots\circ F_n)(I_0),\ldots, x_n(t)\in\hat{F}_{n}^3(\R) \subset F_n(I_0) \quad \text{for all}\  t\ge T'.
$$
The set of last inclusions finalizes the proof of  Lemma \ref{lemma2}.
\hfill $\square$

\vspace{2mm}

With the invariant interval $I_0=[\alpha,M_1]$ for $F$ there is a minimal invariant interval $I_*$ associated with it. It is defined as $I_*=\cap_{n\ge0}F^n(I_0)$. Clearly $I_*\subset F(I_*)$.
The interval $I_*$ is the global attractor for the map $F$ on $\mathbb R$ (it can also become the single fixed point $I_*=\{0\}$).

The next lemma deals with non-oscillatory solutions of systems (\ref{sdds}) and (\ref{sdde}).

\begin{lem}\label{L2}
Suppose that a solution $\mathbf x=(x_1,\ldots,x_n)$ to system (\ref{sdds}) is such that one of its components is non-oscillatory:
$x_k(t)\geq0$ (or $x_k(t)\le0$), $x_k(t) \not\equiv 0,$ for all $t \geq T$ and some $k\in\{1,\ldots,n\}$.
Then all other components of the solution $\mathbf x$ are non-oscillatory as well.
Moreover, there exists $T_1 > T$ such that $|x_j(t)|>0$ for all $t\geq T_1$ and  every $j =1, \dots, n$.
\end{lem}
\begin{proof}
The proof of this lemma uses another simple fact about the solution of the initial value problem (\ref{IVP}) which equally applies to an equivalent IVP:
\begin{equation}\label{IVP0}
u^\prime(t)+\lambda u(t)=c(t),\quad u(t_0)=u_0,\quad \lambda>0.
\end{equation}
\begin{proposition}\label{prop3}
The following statements hold for the solution $u(t)$ of the initial value problem (\ref{IVP0}).\newline
(i) If $u_0>0$ and $c(t)\ge0$ then $u(t)>0$ for all $t\ge t_0$;\newline
(ii) If $u_0<0$ and $c(t)\le0$ then $u(t)<0$ for all $t\ge t_0$;\newline
(iii) If $u_0=0$ and $c(t)\ge0$ (or $c(t)\le0$) and $c(t)\not\equiv0$ then there exists $T>t_0$ such that $u(t)>0$ ($u(t)<0$) for all $t>T$.
\end{proposition}
The validity of the statements of Proposition \ref{prop3} is easily seen from the integral representation of the solution $u(t)$ of (\ref{IVP0}):
$$
u(t)=u_0\,\exp\{-\lambda(t-t_0) \}+\int_{t_0}^tc(s)\exp\{\lambda(s-t)\}\,ds,\quad t\ge0.
$$
In case (iii), $T\ge t_0$ is any point where $c(T)\ne0.$

To be definite, assume that the component $x_k$ satisfies $x_k(t)\ge0$ for all $t\ge T\ge t_0$ and some $T$ (and that $x_k(t)\not\equiv0$ for large $t$).
By using $(k-1)$-st equation of system (\ref{sdds}), $x^\prime_{k-1}(t)+\lambda_{k-1} x_{k-1}(t)=f_{k-1}(x_k(t)):=c(t)\ge0$, the positive feedback of $f_{k-1}$, and Proposition \ref{prop3}, one sees immediately that in the case $x_{k-1}(T)>0$ (or $x_{k-1}(T)=0$) the inequality $x_{k-1}(t)>0$ is satisfied for all $t>T_1$ and some $T_1\ge T$. In the case $x_{k-1}(T)<0$ the component $x_{k-1}(t)$ is strictly increasing in some right neighborhood of $t=T$. If there is a point $T_1>T$ such that $x_{k-1}(T_1)=0$ then one applies part (iii) of Proposition \ref{prop3} to conclude that $x_{k-1}(t)>0$ for all $t\ge T_2$ and some $T_2\ge T_1$.
If such first zero of $x_{k-1}(t), t\ge T,$ does not exist then $x_{k-1}(t)<0, \forall t\ge T.$
Thus in either case $x_{k-1}(t)>0$ or $x_{k-1}(t)<0$ for all sufficiently large $t$. By using next the $(k-2)$-nd equation of system (\ref{sdds}) and applying exactly the same reasoning, one concludes that either $x_{k-2}(t)>0$ or $x_{k-2}(t)<0$ eventually. By going up along equations of system (\ref{sdds}), one arrives at the same conclusion for the first component $x_1(t)$. By using now the last equation of system (\ref{sdds}), and the negative feedback assumption on function $f_n$, one sees that Proposition \ref{prop3} applies to the equation $x^\prime_{n}(t)+\lambda_{n} x_{n}=f_{n}(x_1(t-\tau)):=c(t)$ to conclude that $x_{n}(t)>0$ or $x_{n}(t)<0$ eventually. By continuing up along the system, one obtains that the same definite sign property is valid for the remaining components $x_j, j=k,\ldots, n-1$, and moreover $x_k(t)>0$ eventually.
\end{proof}

\begin{cor}\label{corol2}
All components of any non-oscillating solution $\mathbf x=(x_1,\dots,x_n)$ of system (\ref{sdds})
satisfy \newline $\lim_{t\to\infty}x_j(t)=0$, $j=1,\dots,n$.
\end{cor}
\begin{proof}
Assume first that all the components are positive, $x_j(t)>0, t\ge T, j=1,\dots,n.$ Since $f_n$ satisfies the negative feedback assumption (\ref{nf}) the last equation of system (\ref{sdds}) implies that $x_n$ is eventually decreasing. Set $\lim_{t\to\infty}x_n(t)=x_n^*\ge0$.  By applying Corollary \ref{corol1} one concludes that $\lim_{t\to\infty}x_{n-1}(t)=(1/\lambda_{n-1})f_{n-1}(x_n^*)=:x_{n-1}^*$.
Likewise, $\lim_{t\to\infty}x_{n-2}(t)=(1/\lambda_{n-2})f_{n-2}(x_{n-1}^*)=:x_{n-2}^*$.
In general, for any $k\in\{1,\dots,n-1\}$ one has $\lim_{t\to\infty}x_k(t)=(1/\lambda_{k})f_k(x_{k+1}^*).$ By using the last equation of system (\ref{sdds}) again, one concludes that $\lim_{t\to\infty}x_n(t)=(1/\lambda_n)f_n(x_1^*).$ Therefore, $x_1^*$ is the solution of the iterative equation
$$
x=F(x)=\left(\frac{1}{\lambda_1}f_1\circ \frac{1}{\lambda_2}f_2\circ\ldots\circ \frac{1}{\lambda_n}f_n\right)(x).
$$
Since $F(x)$ satisfies the negative feedback assumption (\ref{nf}), the only solution of the above equation is its fixed point $x=0$.
Therefore, $x_1^*=x_2^*=\ldots=x_n^*=0$.

Assume next that a non-oscillatory solution $\mathbf x$ has both eventually positive and negative components. To be definite, suppose that $x_j(t)>0, j=1,\ldots,k$ and $x_{k+1}(t)<0$ for all sufficiently large $t$ (the case of negative $x_j, 1\le j\le k,$ and $x_{k+1}>0$ is treated analogously). By using the $k$-th equation of the system one sees that $x_k(t)$ is decreasing with $\lim_{t\to\infty}x_k(t)=x_k^*\ge0.$ By using Corollary \ref{corol1} and exactly the same reasoning as above for the case of all positive components, one arrives at the conclusion that $x_k^*$ is the solution (fixed point) of the equation
$$
x=\hat{F}_k(x):=\left(\frac{1}{\lambda_k}f_k\circ \frac{1}{\lambda_{k+1}}f_{k+1}\circ\ldots\circ\frac{1}{\lambda_n}f_n\circ\frac{1}{\lambda_1}f_1\circ\ldots\circ \frac{1}{\lambda_{k-1}}f_{k-1}\right)(x).
$$
Since $\hat{F}_k(x)$ satisfies the negative feedback assumption (\ref{nf}), its only fixed point is $x_k^*=0$, implying that $x_j^*=0$, $j = 1, \dots,n.$
Other subcases are considered similarly.
\end{proof}

\begin{cor}\label{corol3}
If $\mathbf x=(x_1,\dots,x_n)$ is a non-oscillatory solution then each component $x_j, 1\le j\le n,$ satisfies $|x_j^{\,\prime}(t)|>0$ for all $t\ge T$ and some $T>0$.
\end{cor}
\begin{proof}
To prove this we need one more simple fact about the solution $u(t)$ of the initial value problem (\ref{IVP}).
\begin{proposition}\label{prop4}
Suppose that $b(t)$ is strictly decreasing with $b^\prime(t)<0$ and the finite limit\newline
$\lim_{t\to\infty}b(t)=b_*$ exists.
Then the solution $u(t)$  of the initial value problem (\ref{IVP}) is strictly monotone with $|u^{\,\prime}(t)|>0$ for all sufficiently large $t$ with the following specification: \newline
(i) If $u(t_0)\ge b(t_0)$ then $u(t)>b(t)$ and $u^\prime(t)<0$ for all $t\ge t_0$;\newline
(ii) If $b(t_0)>u(t_0)\ge b_*$ then there exists $T>t_0$ such that $u(t)$ is strictly increasing on $[t_0,T]$ and strictly decreasing on $[T,\infty)$.
Moreover $u(t)>b(t)$ and $u^\prime(t)<0$ for all $t>T$;\newline
(iii) If $u(t_0)<b_*$ then either $u(t)$ is strictly increasing with $u^\prime(t)>0$ for all $t\ge t_0$, or there exists a finite $t_1>0$ such that $u(t_1)=b_*$. The latter case is reduced then to part (ii) above.
\end{proposition}
If $u(t_0)>b(t_0)$ then $u^\prime(t_0)<0$. Therefore, $u(t)$ is strictly decreasing in some right neighborhood of $t=t_0$ with $u^\prime(t)<0$ there. We claim that this happens for all $t\ge t_0$. Indeed, one has the inequality $u^\prime(t)<0$ satisfied as long as $u(t)>b(t)$. Suppose that $t_1>0$ is the first time when $u(t_1)=b(t_1)$. Then $u^\prime(t_1)=0$. On the other hand, $u^\prime(t_1)\le b^\prime(t_1)<0$ since $u(t)>b(t)\; \forall t\in[t_0,t_1)$, a contradiction.

If $u(t_0)=b(t_0)$ then $u^{\prime}(t_0)=0$. Since $b^\prime(t_0)<0$ there exists a right neighborhood $(t_0,t_1)$ of $t_0$ such that
$u(t)>b(t)\;\forall t\in(t_0,t_1)$. Exactly by the same reasoning as in the case $u(t_0)>b(t_0)$ above one concludes that $t_1=+\infty$.

If $u(t_0)\in[b_*, b(t_0))$ then $u^\prime(t_0)>0$, so $u(t)$ is strictly increasing in some right neighborhood of $t=t_0$. Since $b(t)$ is decreasing, there exists the first time $t_1>0$ such that $u(t_1)=b(t_1)$. Then $u^\prime(t_1)=0$ and $u(t)>b(t)$ in some open right neighborhood of $t=t_1$. By the very same reason as in the case $u(t_0)>b(t_0)$ above, one concludes that $u(t)>b(t)$ and $u^\prime(t)<0$ for all $t>t_1$.

In the case $u(t_0)<b_*$ one also has $u^\prime(t_0)>0$, so $u(t)$ is increasing in some right neighborhood of $t=t_0$. If $u(t)<b_*\; \forall t\ge t_0$ then $u'(t)>0\; \forall t\ge t_0$. If there exists the first time $t_1$ such that $u(t_1)=b_*$ then the consideration of the case (ii) applies. Therefore $u^\prime(t)<0$ eventually in the latter case.
This completes the proof of Proposition \ref{prop4}.
\hfill $\square$

Analogous statements of the proposition hold valid when $b(t)$ is strictly increasing with $b^\prime(t)>0$ and $\lim_{t\to\infty}b(t)=b_*<\infty$.

Corollary \ref{corol3} now easily follows from Proposition \ref{prop4} applied to equations of system (\ref{sdds}) on each of the steps of the proof of Corollary \ref{corol2}. Indeed, if all the components $x_k(t), 1\le k\le n,$ are positive, then the last equation of the system,
$x_{n}^\prime(t)=-\lambda_n x_n(t)+f_n(x_1(t-\tau)), $  shows that $x_n^\prime(t)<0$ for all sufficiently large $t$. Likewise, since $\lim_{t\to\infty}x_n(t)=0$ by Corollary \ref{corol2}, the second from last equation shows that $x_{n-1}^\prime(t)<0$ for all large $t$. By going up along the system (\ref{sdds}) one concludes that $x_k^\prime(t)<0$ for every $1\le k\le n$ and all sufficiently large $t$. Other subcases are similar.
\end{proof}

For the remainder of the paper, we shall define $\mathbf x$ as a {\em small solution} of system (\ref{sdds}) if
$\displaystyle \lim_{t \to \infty}  \mathbf{ x}(t)e^{\lambda t}=\mathbf0$ for any $\lambda>0$.

\begin{cor}\label{sC}
Under the standing assumption that $f'_k(0)=A_k \not=0, 1\le k\le n,$ system (\ref{sdds}) does not have small non-oscillating solutions.
\end{cor}
\begin{proof} Suppose that $\mathbf{x}(t)$ is a small non-oscillating solution.
As we have proved, in such a case, all $|x_j(t)|$ are strictly decreasing on some interval $[T_2-\tau,+\infty)$ and converge to $0$ as $t \to +\infty$ (see Corollaries \ref{corol2} and \ref{corol3}).  Set $\tau_n =\tau/n$ and consider the function
$$
\hat{ \mathbf x}(t)= (\hat x_1(t), \hat x_2(t), \dots, \hat x_n(t)): = (x_1(t), x_2(t+\tau_n), \dots, x_n(t+(n-1)\tau_n)).
$$
Clearly,  $\hat {\mathbf{x}}(t)$ is a small non-oscillating function as well. Set
$$
\phi(t) = \frac{\hat{ \mathbf x}(t+\tau_n)}{|\hat{ \mathbf x}(t)|},  \quad  t \geq T_2.
$$
We claim that $\inf_{t \geq T_2} |\phi(t)| =0$.  Indeed, otherwise there exists
$\rho \in (0,1)$ such that $$|\hat{ \mathbf x}(T_2+(k+1)\tau_n)| \geq \rho |\hat{ \mathbf x}(T_2+ k \tau_n)| \geq \dots \geq \rho^{k+1} |\hat{ \mathbf x}(T_2)| = e^{(\tau_n^{-1}\ln \rho) (k+1)\tau_n} |\hat{ \mathbf x}(T_2)|, \ k=0,1,2 \dots,$$
contradicting the smallness of $\hat{ \mathbf x}(t)$.

All this proves the existence of a sequence $t_j \to +\infty$ such that $\phi(t_j) \to 0$.  Set {
\begin{equation}\label{Bra}
{ \mathbf y}_j(t) = \frac{\hat{ \mathbf x}(t+t_j+\tau_n)}{|\hat{ \mathbf x}(t_j)|},  \quad  t \geq 0 \quad \left[\mbox{hence,} \ y_{j,k}(t) = \frac{x_k(t+t_j+k\tau_n)}{|\hat{ \mathbf x}(t_j)|}\right].
\end{equation}}
We have that $|{ \mathbf y}_j(-\tau_n)|=1 = \max_{t \geq -\tau_n}|\mathbf {y}_j(t)|$,  and $|\mathbf{y}_j(t)|$ are strictly decreasing functions,  while $\mathbf {y}_j(0)= \phi(t_j) \to 0$ as $j \to +\infty$.  {Also, using the expression in brackets in (\ref{Bra}), it is straightforward to verify that} $\mathbf {y}_j(t) = (y_{j,1}(t), \dots, y_{j,n}(t)), \ t \geq 0,$ satisfies the system
\begin{equation}
\label{2a}
\begin{array}{lll}
y_{j,1}^{\prime}(t) + \lambda_1y_{j,1}(t) & = &a_{1,j}(t)y_{j,2}(t-\tau_n)  \\
y_{j,2}^{\prime}(t) + \lambda_2y_{j,2}(t) & = &a_{2,j}(t)y_{j,3}(t-\tau_n)  \\
\dots & \dots & \dots \\
y_{j,n-1}^{\prime}(t) + \lambda_{n-1}y_{j,n-1}(t) & = &a_{n-1,j}(t)y_{j,n}(t-\tau_n)  \\
y_{j,n}^{\prime}(t) + \lambda_ny_{j,n}(t) & = &a_{n,j}(t)y_{j,1}(t-\tau_n),
\end{array}
\end{equation}
where
$$
a_{k,j}(t)=  \frac{f_k(\hat x_{k+1}(t+t_j))}{\hat x_{k+1}(t+t_j)}.
$$
Since $x_{k+1}(+\infty)=0$ for each $k$,  the functions $a_{k,j}(t)$
are  bounded on $\R_+$ and converging uniformly to $f_k'(0)$ on the positive semi-axis:
 $$\lim_{j \to +\infty}\sup_{t \geq 0}|a_{k,j}(t) - f_k'(0)| =0.$$
Consequently, $\mathbf{y}_j'(t)$ are uniformly bounded on $\R_+$ and therefore  there exists a subsequence $\mathbf y_{j_l}$ such that  $\mathbf y_{j_l}(t)$ converges, uniformly on compact subsets of $\{-\tau_n, -\tau_n/2\} \cup \R_+$, to  a continuous function $\mathbf y(t) = (y_1(t), \dots, y_n(t))$ (in what follows, in order to simplify the notation we will write
$\mathbf y_{j}(t)$ instead of $\mathbf y_{j_l}(t)$).
Note here that the function $\mathbf y(t)$ is defined only on the set $\{-\tau_n, -\tau_n/2\} \cup \R_+$.
Since  each function $|\mathbf{y}_j(t)|$ is strictly decreasing on $[-\tau_n,+\infty)$, we conclude that
$|\mathbf y(t)|$ is  non-increasing on the aforementioned interval. Since $\mathbf{y}(0) = \lim_{j \to +\infty} \mathbf {y}_j(0)= \lim_{j \to +\infty} \phi(t_j) = 0$, this implies that $\mathbf{y}(t) = 0$ for all $t \geq 0$.
In addition, $|\mathbf{y}(-\tau_n)| = 1$.

Due to the Helly theorem \cite[p. 107]{Lax}, we can assume that $y_{j,k}$ $*$-weakly converges on $[-\tau_n,0]$ to some $\phi_k \in L^{\infty}[-\tau_n,0]$. 
Here we are treating $y_{j,k}$  as elements of the standard Lebesgue space  $L^{\infty}[-\tau_n,0]$ of all essentially bounded, real valued, Lebesgue measurable functions on $[-\tau_n,0]$.
When defining $*$-weak convergence in $L^{\infty}[-\tau_n,0]$, we  consider this space as  the dual space of $L^{1}[-\tau_n,0]$.
{As a consequence (see e.g. \cite[Proposition 3.13 (iv)]{BR})}, for each measurable subset $E\subseteq [-\tau_n,0]$ and for each sequence $\{g_j\}$ of continuous functions converging uniformly on $[-\tau_n,0]$ to the function $g$, we have that
$$
\lim_{j \to +\infty}\int_{E}y_{j,k}(s)g_j(s)ds = \int_{E}\phi_{k}(s)g(s)ds.
$$
Therefore, integrating both sides of the $k$-th equation of (\ref{2a}) on $[0,t] \subseteq  [0,\tau_n]$ and then taking the limit as $j \to +\infty$,
we find that
$$
0= \lim_{j\to \infty }\left[y_{j,k}(t) -y_{j,k}(0)   + \lambda_k\int_0^{t}y_{j,k}(s)ds\right] = \lim_{j\to \infty } \int_0^{t}a_{k,j}(s)y_{j,k+1}(s-\tau_n)ds=  f'_k(0)\int^{t}_0\phi_{k+1}(s-\tau_n)ds.
$$
Thus  $\int^{t}_0\phi_{k}(s-\tau_n)ds=0$ for all $t \in [0, \tau_n]$ and therefore $\phi_k =0$ almost everywhere on $[-\tau_n,0]$
for all $k$.  Since $|y_{j,k}(t)|$ are decreasing, this implies that  $y_k(-\tau_n/2)=0$  for every $k$:
$$|y_k(-\tau_n/2)|= \lim_{j \to +\infty} |y_{j,k}(-\tau_n/2)| \leq \lim_{j \to +\infty}\frac{2}{\tau_n}\left|\int_{-\tau_n}^{-\tau_n/2}y_{j,k}(s)ds\right|=\frac{2}{\tau_n}\left|\int_{-\tau_n}^{-\tau_n/2}\phi_{k}(s)ds\right| =0.
$$
Hence, after integrating equations of (\ref{2a}) on $[-\tau_n,0]$, we find that
$$
\lim_{j \to +\infty}\int_{-\tau_n}^0a_{k,j}(s)y_{j,k+1}(s-\tau_n)ds = - \lim_{j \to +\infty} y_{j,k}(-\tau_n).
$$
Recall here that, due to our choice of  $\mathbf y_{j_l}(t)$, the sequences $\{y_{j,k}(-\tau_n)\}$ and  $\{y_{j,k}(-\tau_n/2)\}$ are converging.
Since $|y_{j,k}(t)|$ are positive and decreasing functions
and each $a_{k,j}$ preserves its sign on $[-\tau_n,0]$, this yields that
$$
|y_{j,k}(-\tau_n)|=
\int_{-\tau_n}^0|a_{k,j}(s)| |y_{j,k+1}(s-\tau_n)|ds \geq
$$
$$
\int_{-\tau_n}^{-\tau_n/2}|a_{k,j}(s)| |y_{j,k+1}(s-\tau_n)|ds+  |y_{j,k+1}(-\tau_n)|\int_{-\tau_n/2}^{0}|a_{k,j}(s)| ds.
$$
Taking $\limsup$ on both sides of the above inequality gives
$$
\limsup_{j \to +\infty}\int_{-\tau_n}^{-\tau_n/2}|a_{k,j}(s)| |y_{j,k+1}(s-\tau_n)|ds  +0.5 \tau_n |f'_{k}(0)| |y_{k+1}(-\tau_n)|\leq   |y_{k}(-\tau_n)|.
$$
{On the other hand}, after integrating the equations of (\ref{2a}) on $[-\tau_n,-\tau_n/2]$, we get the following:
$$
\lim_{j \to +\infty}\int_{-\tau_n}^{-\tau_n/2}|a_{k,j}(s)| |y_{j,k+1}(s-\tau_n)|ds = |y_{k}(-\tau_n)|.
$$
Thus $y_{k+1}(-\tau_n)=0$ for every $k$ and, consequently, $\mathbf y(-\tau_n) =0$, which is a contradiction.
\end{proof}

\begin{theorem}\label{T1T}
Assume that each $f_j$ has the first derivative  $f_j'$ which is H\"older continuous with exponent $\alpha$
 in some neighborhood of ~$0$ and that  $f_j'(0) \not=0, j=1,\ldots,n$.  If, in addition, the characteristic equation (\ref{ChE}) has no real  roots, then every solution of system (\ref{sdds}) is oscillating about the zero solution.
\end{theorem}

\begin{proof} On the contrary, suppose that  the characteristic equation has no real  roots and  there
exists a non-trivial solution ${\mathbf x}(t)$ with a non-oscillating component. Then this component is also non-trivial and
thus, in view of Lemma \ref{L2} and its corollaries, all components $|x_j(t)|, j=1,\ldots,n,$ are strictly decreasing and converge
to $0$ at $+\infty$.

With the notation $h_j =0$ for $j=1,2,\dots, n-1$, and $h_n=\tau$, we will compare
the signs of $x_j(t)$ and $ f_j(x_{j+1}(t-h_{j}))$.
Suppose for a moment that $x_j(t)f(x_{j+1}(t-h_j)) > 0$ for all $j$ and large $t$.
Then
$$
0 < x_1(t)f(x_{2}(t-h_1)) x_2(t)f(x_{3}(t-h_2))  \cdot \dots \cdot x_n(t)f(x_{1}(t-h_n)) = \prod_{j=1}^n (x_j(t))^2 \prod_{j=1}^n \frac{f_j(x_{j+1}(t-h_j))}{x_{j+1}(t)} <0,
$$
a contradiction.  Thus there exists some $k$ such that $x_k(t)f_k(x_{k+1}(t-h_k)) <0$ for all large $t$.
This means that
$$
(|x_k(t)|)' + \lambda_k |x_k(t)| <0, \quad t \gg 1.
$$
As a consequence, there exists $t_0$ such that
$$
|x_k(t)| \leq e^{-\lambda_k(t-t_0+h_{k-1})} |x_k(t_0-h_{k-1})|, \quad  t\geq t_0-h_{k-1}.
$$
In particular,  $x_k(t)$ converges exponentially to $0$.
Therefore, since  $|f_{k-1}(x)| \leq 2|f'_{k-1}(0)| |x|$ for all $x$ from some open neighborhood of $0$, one has the estimate
$$
|f_{k-1}(x_k(t-h_{k-1}))| \leq 2|f_{k-1}'(0)|e^{-\lambda_k(t-t_0)} |x_k(t_0-h_{k-1})|, \quad  t\geq t_0.
$$
Set $\lambda= \min\{\lambda_j/2, \ j =1, \dots, n\}$.
After applying the variation of parameters formula to the $(k-1)$-st equation of (\ref{2a}), we obtain that
$$
|x_{k-1}(t)|e^{\lambda t} \leq  e^{(\lambda -\lambda_{k-1})t} \left( |x_{k-1}(t_0)|e^{\lambda_{k-1}t_0}
+ 2|f_{k-1}'(0)|\int_{t_0}^te^{\lambda_{k-1}u}e^{-\lambda_k(u-t_0)} |x_k(t_0-h_{k-1})|du \right) = o(1)
$$
as $t \to +\infty$. This shows that
$x_{k-1}(t)$ is also converging exponentially to $0$.
Clearly, repeating the same argument $n-2$ times more, we find  that  the convergence
${\mathbf x}(t) \to 0, \ t \to +\infty$,  is of   some finite exponential rate $-\gamma \leq - \lambda$, and therefore
$$
f_j(x_{j+1}(t-h_j)) = (f'_j(0) + O(x_{j+1}(t-h_j)^\alpha))x_{j+1}(t-h_j) = (f'_j(0) + O(e^{-\gamma \alpha t}))x_{j+1}(t-h_j), \quad t \to +\infty.
$$
Invoking now \cite[Proposition 7.2]{FA} as well as Lemma \ref{scr}, Claim \ref{prop0} and Corollary \ref{sC}, we conclude that there exist  a root $\lambda_l$  of the characteristic equation  \eqref{ChE}  and real numbers
$b_l:= \Re \lambda_l, c_l:=\Im \lambda_l, D_i, \psi_i$ satisfying  inequalities $b_l \leq - \gamma$,  $|D_1| + \dots + |D_n| >0$,  such that
$$
\mathbf{x}(t) = e^{b_l t}(D_1\cos(c_lt+\psi_1) +o(1), \dots, D_n\cos(c_lt+\psi_n)+o(1)), \ t \to +\infty.
$$
Then, since $c_l=\Im \lambda_l \not=0$,   the solution  $\mathbf{x}(t)$ should oscillate around $0$, a contradiction.
\end{proof}


\section{Existence of nontrivial non-oscillatory solutions}
\label{nosc}

It is well known that even when the  linear system $x'(t) = Lx_t$ of autonomous differential equations with delay  possesses  a non-oscillating solution,  the perturbed asymptotically autonomous  linear system
\begin{equation}
\label{pel}
x'(t)= (L+A(t))x_t, \quad A(+\infty) =0,
\end{equation} can still have all its  solutions oscillating.  In 1995, B. Li provided the following  example of such an equation \cite{BL}:
 $$
 x'(t) + \left(\frac 1 e + \frac{1}{t}\right)x(t-1) =0, \quad t \geq 2.
 $$
 As we will show in this section such irregular behavior of the latter equation is due to the fact that the characteristic equation $z+ e^{-1-z}=0$
 of the associated linear limit  equation has  a unique real (negative) root $z=-1$ of the multiplicity $m_0 = 2$ while the weighted
 perturbation $t^{m_0-1} \cdot (1/t)$ does not belong to  the space $L_1(\R_+)$.  The same reason also explains why this
 phenomenon still persists if the perturbation $t^{-1}$ is replaced with a smaller perturbation  $t^{-2}\sin^2 \pi t$,
  see \cite{ES, Pituk} for further  examples and  references.  Since asymptotically autonomous systems of form  (\ref{pel}) are crucial for the analysis of oscillatory/non-oscillatory  behavior of solutions around the steady states, the above example shows that the ``linearized non-oscillation criterion" for our main system needs a careful justification. We are doing this work below, in the proof of our main result,  Theorem \ref{TTT}.
As a byproduct of our analysis, we obtain the following non-oscillation result for system (\ref{pel}).
\begin{theorem}\label{Tdr}
Suppose that the linear autonomous system $x'(t) = Lx_t$ has an eigenfunction $x(t) =  e^{\lambda t}p $ with
the negative eigenvalue $\lambda$ of multiplicity $m$ and the vector $p = (p_1,p_2,\dots,p_n)$  such that $p_j \not =0$ for all $1\le j\le n$.
Suppose also that there are no complex eigenvalues with the same real part $\lambda$  and that $|A(t)| \leq q(t), \ t \geq t_0,$ where $q$ is a decreasing function such that $q(+\infty)=0$ and $q(t)t^{m-1}$ is
integrable on $[t_0, +\infty)$. Then equation (\ref{pel}) has a nonzero non-oscillating solution.
\end{theorem}
The proof of Theorem~\ref{Tdr} essentially mimics  the main part of the proof of Theorem \ref{TTT} and therefore it is given in Appendix.
\begin{theorem}\label{TTT}
Assume that each $f_j$ is $C^{2}$-smooth  in some neighborhood of \ $0$,  $f_j'(0) \not=0$, and the characteristic equation (\ref{ChE}) has a real (hence, negative)  root $\lambda_-$. Then  system (\ref{sdds})  has at least one non-oscillating non-zero solution on some interval $[T, +\infty)$.
\end{theorem}
\begin{proof} We can assume that $\lambda_-$ is the largest real root of multiplicity $m \geq 1$. Due to Lemma \ref{scr}, each complex eigenvalue $\mu_j$ with $\Re \mu_j > \lambda_-$ is simple.  Note that there exists a unique eigenvalue having $\lambda_-$ as the real part (which is
$\lambda_-$ itself).

With $\Lambda$ denoting the diagonal matrix $\Lambda:= {\rm diag}\{\lambda_1, \dots, \lambda_n\}$, the nonlinear system can be written as
\begin{equation}
\label{3na}
y'(t) + \Lambda y(t) = L(y_1(t-\tau), y_2(t),\dots,y_n(t)) + N(y_1(t-\tau), y_2(t),\dots,y_n(t)),
\end{equation}
where its linearization about $0$ is of the form $y'(t) + \Lambda y(t) = L(y_1(t-\tau), y_2(t),\dots,y_n(t))$:
\begin{equation}
\label{3la}
\begin{array}{lll}
y_{1}^{\prime}(t) + \lambda_1y_{1}(t) & = &f'_{1}(0)y_{2}(t),  \\
y_{2}^{\prime}(t) + \lambda_2y_{2}(t) & = &f'_{2}(0)y_{3}(t),  \\
\dots & \dots & \dots \\
y_{n-1}^{\prime}(t) + \lambda_{n-1}y_{n-1}(t) & = &f'_{n-1}(0)y_{n}(t),  \\
y_{n}^{\prime}(t) + \lambda_ny_{n}(t) & = &f'_{n}(0)y_{1}(t-\tau).
\end{array}
\end{equation}
The higher order term $N(u)= N(u_1,\dots,u_n)$ satisfies
\begin{equation}\label{NeS}
|N(u)| \leq K |u|^{2}, \ |N(u)-N(v)| \leq K(|u|+|v|) |u-v|, \ \ |u|, |v| \leq r, \quad \mbox{for some} \  r, K>0.
\end{equation}
When $y(s) =  e^{\lambda_- s}p = e^{\lambda_- s}(p_1, \dots, p_n)$ is an eigenfunction for system (\ref{3la}),
then the following holds
$$
(\lambda_- + \lambda_j)p_j = f'_j(0)p_{j+1}, \ j=1,\dots,n-1; \quad (\lambda_- + \lambda_n)p_n = f'_n(0)p_{1}e^{-\lambda_- \tau}.
$$
In particular, all $p_j, \ j =1, \dots, n,$ are non-zero real numbers. From now on, we will fix them, and will assume that $r$ is small enough
compared with $p$, e.g. $r < 0.1 \min |p_j|$.

The  change of variables $y(t) = e^{\lambda_- t}(p +z(t))$ transforms (\ref{3na}) into
\begin{equation}
\label{3nac}
z'(t) + (\Lambda +\lambda_- I)z(t) = L(z_1(t-\tau)e^{-\lambda_-\tau}, z_2(t),\dots,z_n(t)) + M_1(t, z_1(t-\tau), z_2(t),\dots,z_n(t)),
\end{equation}
where
$$
M_1(t, z_1, z_2,\dots,z_n) = e^{-\lambda_- t}N(e^{\lambda_-(t-\tau)}(p_1+z_1), \dots,e^{\lambda_-t}(p_n+z_n)).
$$
Clearly, the eigenvalues of the linear system
\begin{equation}
\label{3lac}
z'(t) + (\Lambda +\lambda_- I)z(t) = L(z_1(t-\tau)e^{-\lambda_-\tau}, z_2(t),\dots,z_n(t))
\end{equation}
can be obtained from the eigenvalues of (\ref{3la}) by adding $-\lambda_-$.  In particular, they contain $0$ and
a finite set {\color{red} $\{\mu_j: j \in J\subset\mathbb N\}$} (possibly, empty) of simple complex eigenvalues $\mu_j$ with positive real parts:
$\mu_j = \alpha_j + i\beta_j$,  $\alpha_j >0$, $\beta_j \not=0$.
This implies that the fundamental matrix solution $Z(t), \ t \geq 0,$ for system (\ref{3lac}) can be written in the form
$$
Z(t) = \sum_{j\in J} e^{\mu_j t}P_j  + Q(t) + Z_-(t),
$$
where $P_j$ are $n\times n$- complex valued matrices, $Q(t)$ is a matrix polynomial of the degree $m-1$ and
$|Z_-(t)| \leq de^{-\gamma t}, t \geq 0$, for some positive $d$ and $\gamma$.
We recall that the fundamental matrix solution solves the initial value problem $Z(s)=0, \ s \in [-\tau,0), \ Z(0) = I$ for system (\ref{3lac}), see \cite{Hale}.  Note also that each matrix-valued function from the set $\{\sum_{j\in J} e^{\mu_j t}P_j, t \in \R; \ Q(t), t \in \R;  \ Z_-(t), t \geq 0\}$ is an individual  solution of (\ref{3lac}).
See \cite[Chapter 6]{BC} (and especially Section 6.7 (Exercise 3) and Section 6.8 (Theorem 6.7 and Exercise 1)) for more details concerning the above listed properties of $Z(t)$.\footnote{In \cite{BC}, the fundamental matrix  is called the kernel matrix  and the notation  $K(t)$ is used there  instead of $Z(t)$.}
System (\ref{3lac}) can be written in a shorter form $z'(t) = {\mathcal L}z_t$ (or $z'(t)= {\mathcal L}z(t+\cdot))$ where the operator ${\mathcal L}:C([-\tau,0],\R^n) \to \R^n$ and
$z(t+\cdot)=z_t \in C([-\tau,0],\R^n)$ are defined by $z_t(s) = z(t+s), s \in [-\tau,0]$, and
$$
{\mathcal L}\phi =  -(\Lambda +\lambda_- I)\phi(0) + L(\phi_1(-\tau)e^{-\lambda_-\tau}, \phi_2(0),\dots,\phi_n(0)), \ \phi(s) = (\phi_1(s), \dots, \phi_n(s)).
$$
Set now {$K_0:= 4K(|p|+|p|^2)e^{-2\lambda_-\tau}$} and take {$\rho < \min \{0.25, r\}$.} Let  $\sigma$ be sufficiently large to satisfy {$|u+p|e^{\lambda_- (t-\tau)} \leq r$ } for $t \geq \sigma$, $|u| \leq r$, as well as
\begin{equation}\label{K0}
K_0\left(  \sum_{j\in J} \int_t^{+\infty }e^{\Re \mu_j (t-s)}\|P_j\| e^{\lambda_- s}ds + \int_t^{+\infty } \|Q(t-s)\| e^{\lambda_- s}ds   +  \int_\sigma^t\|Z_-(t-s)\|  e^{\lambda_- s}ds \right) < \rho < r, \ t \geq \sigma.
\end{equation}
{For $t \geq \sigma$ and $|u|, |v| \leq r < |p|$, we also have from (\ref{NeS}) that
\begin{eqnarray*}
& & |M_1(t, u_1, u_2,\dots,u_n)-  M_1(t, v_1, v_2,\dots,v_n)|
\\
& &= e^{-\lambda_- t}|N(e^{\lambda_-(t-\tau)}(p_1+u_1), \dots,e^{\lambda_-t}(p_n+u_n))-  N(e^{\lambda_-(t-\tau)}(p_1+v_1), \dots,e^{\lambda_-t}(p_n+v_n))|
\\
& &
\leq 4K|p|e^{-2\lambda_-\tau}e^{\lambda_- t}|u-v| \leq K_0e^{\lambda_- t}|u-v|,
\\
& & |M_1(t, u_1, u_2,\dots,u_n)| \leq Ke^{-\lambda_- t}|(e^{\lambda_-(t-\tau)}(p_1+u_1), \dots,e^{\lambda_-t}(p_n+u_n))|^{2}
\\
& &  \leq
 Ke^{ \lambda_- t}|(e^{-\lambda_-\tau}(p_1+u_1), \dots,p_n+u_n)|^{2} \leq  4Ke^{ \lambda_- t}e^{-2\lambda_-\tau}|p|^{2} \leq  K_0e^{ \lambda_- t}.
\end{eqnarray*}}
As we can see, the nonlinearity $M_1$ is uniformly  dominated by an integrable function on $\R_+$.

 {Our next goal is to prove that, for sufficiently large $\sigma$,  system (\ref{3nac}) has a solution $z:[\sigma+\tau, +\infty) \to \R^n$ such that $|z_j(t)| \leq 0.1 \min_{k}|p_k|$,  $t \geq \sigma+\tau$,  for each $j$.}  Clearly, this implies that the solution $y(t) = e^{\lambda_- t}(p +z(t)),$  {$ t > T$,}  of the original system (\ref{3na}) does not oscillate
around $0$.

To this end, for $t \geq \sigma$, consider the integral equation
\begin{equation}\label{cal}
z(t) = \mathcal A z(t): =  -\sum_{j\in J} \int_t^{+\infty }e^{\mu_j (t-s)}P_j M_1(s)ds - \int_t^{+\infty} Q(t-s)  M_1(s)ds
 + \int_\sigma ^t Z_-(t-s) M_1(s)ds, 
\end{equation}
where the abbreviation $M_1(s)= M_1(s, z_1(s-\tau), z_2(s),\dots,z_n(s))$ is used.
We will assume that  $z(s) \equiv 0$ on $[\sigma-\tau, \sigma)$.
Suppose that this equation has a solution $z_*: [\sigma, +\infty) \to \{u: |u|\leq r\}$ (observe that we do not claim that $z_*(\sigma) =0$, this property is not relevant for our further analysis).  Then it can be verified that $z_*(t), \ t \geq \sigma+ \tau$, also satisfies system (\ref{3nac}) (alternatively,  this fact  can be deduced from considerations in   \cite[p. 172]{Hale}).
Indeed, after differentiating equation (\ref{cal}) with respect to $t \geq \sigma+\tau$, we obtain  that
\begin{eqnarray*}
z'(t) &=& (\sum_{j\in J} P_j +Q(0)+Z_-(0))M_1(t)  -\sum_{j\in J} \int_t^{+\infty }(e^{\mu_j (t-s)}P_j)' M_1(s)ds\\
      &-& \int_t^{+\infty} Q'(t-s)  M_1(s)ds + \int_\sigma ^t Z_-'(t-s) M_1(s)ds\\
      &=& M_1(t)  -\sum_{j\in J} \int_t^{+\infty }({\mathcal L}e^{\mu_j (t+\cdot-s)}P_j) M_1(s)ds\\
      &-& \int_t^{+\infty} ({\mathcal L}Q(t+\cdot-s))  M_1(s)ds + \int_\sigma^t ({\mathcal L}Z_-(t+\cdot-s)) M_1(s)ds\\
      &=& M_1(t)  + {\mathcal L}\left(-\sum_{j\in J} \int_{t+\cdot}^{+\infty }e^{\mu_j (t+\cdot-s)}P_j M_1(s)ds\right.\\
      &-& \left.\int_{t+\cdot}^{+\infty} Q(t+\cdot-s)  M_1(s)ds + \int_\sigma ^{t+\cdot} Z_-(t+\cdot-s) M_1(s)ds\right)\\
      &=& {\mathcal L}z_t+ M_1(t).
\end{eqnarray*}
In the fist line  of the above chain of equalities, we invoke the relation $\sum_{j\in J} P_j +Q(0)+Z_-(0) = Z(0) =I$. Next, in  order to prove the equality
in the fifth line, we use another property of the fundamental matrix solution:  $Z(s)=0$ for $s \in [-\tau,0)$.

Now, take a bounded continuous function $u(t)$, $|u(t)| \leq r,  \ t \geq \sigma,$ and set $u(t) \equiv 0$ for $t < \sigma$.
The set $\frak B$ of all such functions $u$ equipped with the distance $\rho(u,v) = \sup_{s \geq \sigma}|u(s)-v(s)|$ is a complete metric space.
Then, for all $ t \geq \sigma,$ one has
$$
|\mathcal A u(t)| \leq  K_0\left(  \sum_{j\in J} \int_t^{+\infty }e^{\Re \mu_j (t-s)}\|P_j\| e^{\lambda_- s}ds + \int_t^{+\infty } \|Q(t-s)\| e^{\lambda_- s}ds   +  \int_\sigma^t\|Z_-(t-s)\|  e^{\lambda_- s}ds \right) < \rho < r.
$$
Thus the operator ${\mathcal A}: {\frak B} \to \frak{B}$ is well defined. Moreover, ${\mathcal A}$ is a contraction in view of the following relation:
{
$$
|\mathcal A v(t)- \mathcal A u(t)| \leq  \rho \sup_{t \geq \sigma} |\left(v_1(t-\tau)- u_1(t-\tau), \dots, v_n(t)-u_n(t)\right)| \leq 0.5 \sup_{t \geq \sigma} |v(t)- u(t)|,
 \ t \geq \sigma.
$$
}In this way,  the Banach contraction principle guarantees the existence of the required $z_*\in \frak B$.
\end{proof}

\begin{cor}
For system (\ref{sdds}) with any fixed $\tau>0,$ positive $\lambda_1,\lambda_2, \dots, \lambda_n$ and $a=-A_1\cdot A_2\cdot\ldots\cdot A_n>0$ exactly one of the following two options is possible:
\begin{enumerate}
\item[(i)]
for any $a>0$, all its solutions oscillate about zero;
\item[(ii)]
there exists $a_0>0$ such that for any $a$ satisfying $0<a<a_0$ system \eqref{sdds} has a non-oscillating solution, while for $a>a_0$ all its solutions oscillate about zero.
\end{enumerate}
\end{cor}
\begin{proof}
The application of Lemma~\ref{lem_real} and Theorems \ref{T1T}, \ref{TTT} implies the statement of the corollary.
\end{proof}

\section{Discussion}
\label{disc}

\subsection{Monotonicity, oscillation, and instability}

It is a well known fact that  the value of $a_1$ in Lemma \ref{pure imagi} defines the stability boundary for both systems (\ref{sdds}) and (\ref{ldds}). For any $0<a<a_1$, the zero solution $(x_1,\dots,x_n)=(0,\dots,0)$ is asymptotically stable, while for all $a>a_1$ it is unstable. The value $a=a_1$ is critical for system (\ref{sdds}). While linear system (\ref{ldds}) is still stable (but not asymptotically stable), nonlinear system (\ref{sdds}) can be either stable or unstable depending on specific local shapes of nonlinearities $f_j, 1\le j\le n$ around zero.

It is an interesting question of relative location of the values of $a_1$ from Lemma \ref{pure imagi} and $a_0$ from Lemma \ref{lem_real} (when the latter exists). By Corollary \ref{corol1a},
in the  scalar $n=1$ and the two-dimensional $n=2$ cases we have that $a_0<a_1$  (see also  \cite{adH79a,GyoLad91,HadTom77}).
This in particular means that the instability of the zero solution in either system (\ref{sdds}) or system (\ref{ldds}) implies that all their solutions oscillate.
However, it is not the case for the general $n$. When $n=3$ both possibilities $a_0<a_1$ or $a_0>a_1$ can happen for system (\ref{ldds}) or (\ref{sdds}) \cite{IvaBLW04}. This in particular shows that the zero solution to system (\ref{ldds}) can be unstable, and still the system can have a non-oscillatory exponentially decaying to zero solution.

In this section, for every $n\ge4$ we construct an example of characteristic equation (\ref{ChE}) which has at least  one negative root and at least one pair of complex conjugate roots with  positive real part. In our construction we use the following simple perturbation argument. Suppose that an analytic function $F(z)$ has exactly $m$ zeros $z_1,\ldots,z_m$ inside a connected bounded domain $D\subset\mathbb C$ with simple boundary (counting with the multiplicities). Then there exists $\varepsilon_0>0$ such that for every analytic function $F_1(z)$ satisfying $|F_1(z)|\le\varepsilon_0\; \forall z\in D$ the perturbed analytic function $F(z)+F_1(z)$ also has exactly $m$ zeros on the domain $D$. Moreover, if $w_1,\ldots,w_m$ are respective zeros of $F(z)+F_1(z)$ with $\max_{j}|z_j-w_j|=\varepsilon_1$ then $\varepsilon_1\to0$ as $\varepsilon_0\to0$. This statement can be proved by using Rouch\'e's Theorem (for further details and proof see e.g. \cite{Con78}).

Let $k\ge 3$ be given. The polynomial $P_n(z)=(z+1)z^k$ of degree $n=k+1$ has the negative root $z_0=-1$ and the zero root $z_1=0$ of multiplicity $k$.
For the perturbed polynomial $P_{n,p}(z)=(z+1)z^k+p$, there exists $p_0>0$ such that for every $0<p\le p_0$ the following properties hold:
\begin{itemize}
\item[(i)]
there exists a negative root $z_{0,p}$ such that $|z_{0,p}+1|\le\varepsilon_0$ with $\varepsilon_0\to 0$ as $p\to0$;
\item[(ii)] there exists a pair of complex conjugate roots $z_{1,2}=\alpha\pm i\,\beta$ with  positive real part $\alpha>0$.
\end{itemize}
The validity of part (i) is seen immediately. Since $z_1$ and $z_2$ in part (ii) are close to zero for small $p$, they are also close to the roots of the polynomial  $z^k+p$. Therefore they are close to the $k$th roots of $p(-1),$ $z_l=\sqrt[k]{p}\,\exp\{i\pi(2l-1)/k\}, 1\le l\le k,$ which always contain a conjugate pair with  positive real part for every $k\ge3$.

Consider next the analytic function
$$
F=F(z,p,\tau,k,\lambda_1,\dots,\lambda_k)=(z+1)(z+\lambda_1)\cdot\ldots\cdot(z+\lambda_k)+p\,\exp\{-\tau\, z\}.
$$
If $\tau=\lambda_1=\ldots=\lambda_k=0$ then $F=P_{n,p}(z)=(z+1)z^k+p$. Therefore if $p>0$ is such that $P_{n,p}(z)$ has a negative root and a pair of complex conjugate roots with positive real part then there exist $\lambda_0>0$ and $\tau_0>0$ such that for all $0\le\lambda_l\le\lambda_0, 0\le l\le k,$
and $0\le\tau\le\tau_0$, the analytic function  $F(z,p,\tau,k,\lambda_1,\dots,\lambda_k)$ also has a nearby negative root and a nearby pair of complex roots with  positive real part. Therefore, we have shown the following
\begin{proposition}\label{prop6}
There exists $p_0>0$ such that for every $p\in(0,p_0]$ there are $\lambda_0>0$ and $\tau_0>0$ such that for an arbitrary choice of
$\lambda_l\in[0,\lambda_0], 1\le l\le k,$ and $\tau\in[0,\tau_0]$, the corresponding characteristic equation (\ref{ChE})
$$
(z+1)(z+\lambda_1)\cdot\ldots\cdot(z+\lambda_k)+p\,\exp\{-\tau\, z\}=0
$$
has a negative real solution and a pair of complex conjugate solutions with positive real part.
\end{proposition}

The case $n=3$ of the characteristic equation (\ref{ChE}), which is not covered by Proposition \ref{prop6}, was treated in detail in paper \cite{IvaBLW04}.
The possibility for it to have a negative root and a pair of complex conjugate roots with positive real parts was shown there in a general setting
(see \cite[Lemma 2.3 and Remark 2.4]{IvaBLW04}). An alternative example can be given by the following equation
$$
(z+1)(z+0.5)(z+0.25)+2.125\,\exp\{-z\,\tau\}=0,
$$
which has such a triplet of solutions for all sufficiently small values of $\tau, 0<\tau\le\tau_0$ for some $\tau_0>0$. The numerical approximation shows the following outcome:
\begin{itemize}
\item[(i)] $\tau=0.5, z_1 \approx -14.1259, z_2 \approx -2.5878, z_{3,4} \approx 0.2206 \pm 0.9359\,i;$
\item[(ii)] $\tau=0.25, z_1 \approx -41.5326, z_2 \approx -2.1551, z_{3,4} \approx 0.1637 \pm 1.0118\,i$;
\item[(iii)] $\tau=0.1, z_1 \approx -140.7464, z_2 \approx -1.9947, z_{3,4} \approx  0.1167 \pm 1.0558\,i$,
\end{itemize}
with the $\tau_0$ estimated at $\tau_0\approx 0.741005$. For this equation one also has that the parameter value $a_0$ belongs to the interval  $a_1<a_0<a_2$.

\subsection{Periodic solutions}

The problem of existence of periodic solutions to particular cases of system (\ref{sdds}) has been addressed in multiple publications. The scalar equation was studied by Hadeler and Tomiuk in \cite{HadTom77}. The two-dimensional case was treated by an der Heiden in \cite{adH79a} for a scalar second order nonlinear delay equation. The three dimensional case was studied  by Ivanov and Lani-Wayda in \cite{IvaBLW04}.
Two versions of the general $n$-dimensional case were considered in papers \cite{HalIva93,Mah80}  (which are essentially higher order scalar nonlinear equations in both papers). Hale and Ivanov viewed it from the singular perturbation point of view, while Mahaffy assumed certain parameters of the system to be large. The knowledge about existence of periodic solutions for general system (\ref{sdds}) remains incomplete.

Proofs of the existence of periodic solutions in all the cases mentioned above use principal ideas of the ejective fixed point techniques applied to appropriately constructed cones on the space of initial functions (see e.g. monographs \cite{DieSvGSVLWal95,Hale} for details of the ejective fixed point theory). In part it requires the leading eigenvalue $z=\alpha+i\,\beta, 0<\beta<\pi/\tau$ to have positive real part $\alpha>0$. Additionally the assumption that this eigenvalue is unique is also used. We shall show by the following example that in general such uniqueness is not guaranteed, and there can be multiple eigenvalues with  positive real parts and the imaginary parts between $0$ and $\pi/\tau$.

\begin{proposition}\label{prop7}
Consider characteristic equation (\ref{ChE}) for $a>0$ and $\tau>0$. Let $a_k, k\in\mathbb N,$ be the values of parameter $a$
as defined by Lemma \ref{pure imagi}.
\begin{itemize}
\item[(i)] For any $n\le4$ and all $a>a_1$, equation (\ref{ChE}) has one and only one eigenvalue $\alpha+i\,\beta$ with $\alpha>0$ and $0<\beta<\pi/\tau$;
\item[(ii)] For every $n\ge5$ with $4m-3\le n\le 4m$ for some $m\ge2$, there exists an example of  equation (\ref{ChE}) such that it has $m$ pairs of complex conjugate roots $\alpha_j\pm i\,\beta_j, 1\le j\le m,$ with $\alpha_j>0$ and $0<\beta_j<\pi/\tau$ for all $a\in(a_m,a_m^*)$, where $a_m^*$ is a finite value dependent on $\lambda_j, 1\le j\le n$. Their real and imaginary parts are ordered as follows
    $$
    0<\alpha_m<\cdots<\alpha_1;\quad 0<\beta_1<\cdots<\beta_m<\pi/\tau;
    $$
\item[(iii)] Given equation (\ref{ChE}) with arbitrary $n\ge1$, there exists a finite $a^*\ge a_1$ such that for all $a>a^*$ it has one and only one eigenvalue  $\alpha+i\,\beta$ with $\alpha>0$ and $0<\beta<\pi/\tau$.
    It is the leading eigenvalue $\alpha_1+i\,\beta_1$ which imaginary part is $\omega_1$ at the parameter value $a=a_1$.
\end{itemize}
\end{proposition}
\begin{proof}
We refer here to Lemma \ref{pure imagi} and its proof for the notations and related facts established there.

(i)\,
Recall that function $\Theta_0(\omega)$ is increasing, concave down, and satisfying $\Theta_0(0)=0, \Theta_0(+\infty)=n\pi/2$.
For the bifurcation value $\omega_1, 0<\omega_1<\pi/\tau,$ consider the corresponding eigenvalue $\alpha_1(a)+i\,\beta_1(a)$.
By Lemma \ref{pure imagi}, (iii), $\alpha_1(a)>0$ and $0<\beta_1(a)<\pi/\tau$ for all $a\ge a_1$.
When $n\le4$, the second bifurcation value $\omega_2$ satisfies $\pi/\tau<\omega_2<3\pi/\tau$ (as it is found from the intersection of the graphs of $y=\Theta_0(\omega)$ and $y=-\omega\tau+3\pi$).
For the corresponding eigenvalue $\alpha_2(a)+i\,\beta_2(a)$ its imaginary part $\beta_2(a)$ is increasing for $a>a_2$ with $\lim_{a\to\infty}\beta_2(a)=3\pi/\tau.$
Since the sequence $\omega_k, k\in\mathbb N,$ is increasing there are no other eigenvalues $\alpha_k(a)+i\,\beta_k(a), k\ge2,$ with $0<\beta_k(a)<\pi/\tau$ for $a\ge a_k$.

(ii)\,
Let $n$ satisfy $5\le n\le8$, so that $m=2$. Since $\lim_{\omega\to\infty}\Theta_0(\omega)=n\pi/2>2\pi$, there is a possibility in this case that the graph of $y=\Theta_0(\omega)$ intersects with the line $y=3\pi-\omega\tau$ at a value $\omega_2$ satisfying $\omega_2<\pi/\tau$ (see Fig. 1, $n=6$). This would give a second pair of eigenvalues $\alpha_2(a)\pm i\,\beta_2(a)$ with $\alpha_2(a)>0$ and $0<\beta_2(a)<\pi/\tau$ in some right neighborhood of $a_2$, $(a_2,a_2^*)$.
Such second pair will always exist when $\Theta_0(\pi/\tau)>2\pi$.
This situation can be achieved for example when $\Theta_0^\prime(0)=\sum_{j=1}^{n}(1/\lambda_j)$ is made sufficiently large.
One can easily see that there exists sufficiently small $\lambda_0>0$ such that for any choice of $\lambda_j\in(0,\lambda_0), 1\le j\le n,$
the characteristic equation (\ref{ChE}) has the required second pair of complex conjugate solutions.
From Lemma \ref{pure imagi} (iii) and Claim \ref{prop0} it follows that the inequalities $\alpha_1(a)>\alpha_2(a)>0$ and $0<\beta_1(a)<\beta_2(a)$ are preserved for all $a\ge a_2.$ Besides one also has that $\beta_2(a)<\pi/\tau$ for all $a\in(a_2, a_2^*)$ and some $a_2^*>a_2$. Since $n\le 8$ all other values $\omega_k, k\ge 3,$ satisfy $\omega_3>\pi/\tau$. Therefore, in the case $m=2$ equation (\ref{ChE}) can have no more than two complex conjugate roots with the imaginary part in the interval $(0,\pi/\tau)$.

The general case of $4m-3\le n\le 4m$ for some $m\ge3$ is treated similarly to the case $m=2$ above. Since $\lim_{\omega\to\infty}\Theta_0(\omega)=n\pi/2>2(m-1)\pi$ the first $m$ intersections of $y=\Theta_0(\omega)$ with the lines $y=-\omega\tau+\pi(2k-1)$ can all satisfy $0<\omega_k<\pi/\tau, 1\le k\le m$. This will be the case when the inequality $\Theta_0(\pi/\tau)>2(m-1)\pi$ holds. The latter can be achieved when $\Theta^\prime(0)$ is made large enough. Which in turn will be the case  when all $\lambda_j, 1\le j\le n,$ are small enough, $\lambda_j\in (0,\lambda_0)$ for some $0<\lambda_0\ll1$. For the respective complex conjugate roots $\alpha_k(a)\pm i\beta_k(a)$ of the characteristic equation (\ref{ChE}) the order
$$
\alpha_1(a)>\alpha_2(a)>\dots \alpha_m(a)\quad\text{and}\quad 0<\beta_1(a)<\beta_1(a)<\dots <\beta_m(a)<\frac{\pi}{\tau}
$$
will be preserved for all $a\in(a_m,a_m^*)$ and some $a_m^*>a_m$. Since $\omega_{m+1}>\pi/\tau$ there are no other eigenvalues with the imaginary part within the interval $(0,\pi/\tau)$.  As an illustration see Fig. 1 for the case $n=10\, (m=3)$ where the inequalities $0<\omega_1<\omega_2<\omega_3<\pi/\tau$ and the inclusion $\omega_4\in (\pi/\tau,3\pi/\tau)$ hold.

(iii)\,
As it is shown in part (ii) above for every fixed $n$ such that $4m-3\le n\le4m$ for some positive integer $m$ there can be only a maximum $m$ of eigenvalues with the imaginary part between $0$ and $\pi/\tau$. For the first eigenvalue $\alpha_1(a)+i\,\beta_1(a)$ one has that the inequality $0<\beta_1(a)<\pi/\tau$ holds for all $a\ge a_1$. Since $\beta_j(a)$ is increasing in $a>0$ and $\lim_{a\to\infty}\beta_j(a)=\pi(2j-1)/\tau, j\in\mathbb N,$ one sees that there exists $a^*\ge a_1$ such that $\beta_j(a)$ satisfies $\beta_j(a)>\pi/\tau$ for all $a\ge a^*$ and $j\ge2$.
\end{proof}

Our numerical calculations show two pairs of complex conjugate solutions with  positive real part
and the imaginary part within $(0,\pi/\tau)$ for the following values of the parameters:
$$
n=5,\quad \tau=1,\quad \lambda_k=1/2^k, \, k=1,2,\dots,5, \quad a=200,
$$
$$
z_{1,2} \approx 1.5456 \pm 0.9058\,i,  \quad z_{3,4} \approx 0.2221\pm 2.6734\,i.
$$

For the range $9\le n\le 12$ one can get three pairs of complex conjugate solutions with the same as above properties.
Our calculations show the following outcome:
$$
n=9,\quad \tau=1,\quad \lambda_k=1/2^k, \, k=1,2,\dots,9, \quad a=10^4,
$$
$$
z_{1,2} \approx 1.9499 \pm 0.6159\,i, \quad z_{3,4} \approx 1.3703 \pm 1.7812\,i, \quad z_{5,6} \approx 0.1272 \pm 2.7047\,i.
$$
In view of the considerations in this subsection we are in a position to state the following conjecture.

\begin{conjecture}\label{conj1}
Assume all the hypotheses on the nonlinearities $f_j, 1\le j\le n,$ as stated in the Introduction.
Suppose in addition that $a>\max\{a_0, a_1\}$, where $a_0$ and $a_1$ are defined in Lemmas \ref{lem_real} and \ref{pure imagi}.
Then system (\ref{dds}) has a nontrivial slowly oscillating periodic solution.
\end{conjecture}

The slow oscillation here is meant in the sense of papers \cite{HalIva93,Mah80,mps,mps2}. Every component $x_j, 1\le j\le n,$ is slowly oscillating, and there is one and only one zero of any other component $x_l, l\ne j,$ in between any two consecutive zeros of the component $x_j$.

The conjecture is a proven fact for the cases $n=1, 2$ and $3$, which was established in papers \cite{HadTom77},\cite{adH79a} and \cite{IvaBLW04}, respectively. It is an open problem for the general case $n\ge4$.
\vspace{0.2cm}

\subsection*{Appendix: Proof of Theorem \ref{Tdr}.}
The  change of variables $y(t) = e^{\lambda  t}(p +z(t))$ transforms (\ref{pel}) into
\begin{equation}
\label{3nacAp}
z'(t) + \lambda z(t) = L[e^{\lambda \cdot}z(t+\cdot)]+M_1(t,z(t+\cdot)),
\end{equation}
where
$$
M_1(t,z(t+\cdot))=
 A(t)[e^{\lambda\cdot}(p+z(t+\cdot))].
$$
Fix some $\sigma > t_0+2\tau$ and $r \in (0, 0.1\min\{|p_j|, j =1, \dots,n\})$ and observe that, for $|u(t)|, |v(t)| \leq r$, $t \geq \sigma-\tau$, it holds
$$
|M_1(t,u(t+\cdot))|\leq
|A(t)||e^{\lambda\cdot}(p+u(t+\cdot))|_C\leq  2|p|q(t)e^{-\lambda\tau} \leq   (2|p|+1)e^{-\lambda\tau}q(t) =:K_0q(t);
$$
$$
|M_1(t,u(t+\cdot))- M_1(t,v(t+\cdot))|\leq
|A(t)||e^{\lambda\cdot}(u(t+\cdot)- v(t+\cdot))|_C\leq  q(t)K_0|u(t+\cdot)- v(t+\cdot)|_C.
$$
Here we are using the usual notation $|w(t+\cdot)|_C = \sup\{|w(t+s)|, \ s \in [-\tau,0]\}.$
As we can see, the nonlinearity $M_1$ is uniformly  dominated by an integrable function on $\R_+$.

Clearly, the eigenvalues of the linear system
\begin{equation}
\label{3lacAP}
z'(t) + \lambda z(t) = L[e^{\lambda \cdot}z(t+\cdot)]
\end{equation}
(or, equivalently, $z'(t) = {\mathcal L}z_t$, where ${\mathcal L} \phi =  - \lambda \phi(0)+ L[e^{\lambda \cdot}\phi(+\cdot)]$ and $z_t(s) = z(t+s)$)
can be obtained from the eigenvalues of $x'(t)=Lx_t$ by adding $-\lambda$.  In particular, they contain $0$ and
a finite set {\color{red} $\{\mu_j: j \in J\subset\mathbb N\}$} (possibly, empty) of complex eigenvalues $\mu_j$ with positive real parts:
$\mu_j = \alpha_j + i\beta_j$,  $\alpha_j >0$, $\beta_j \not=0$.
This implies that the fundamental matrix solution $Z(t), \ t \geq 0,$ for system (\ref{3lacAP}) can be written as
$$
Z(t) = \sum_{j\in J} e^{\mu_j t}P_j(t)  + Q(t) + Z_-(t),
$$
where $P_j(t)$ are complex matrix polynomials, $Q(t)$ is a real matrix polynomial of degree $m-1$ and the inequality
$|Z_-(t)| \leq de^{-\gamma t}, t \geq 0$, holds for some positive $d$ and $\gamma$.

Take now {$\rho < \min \{0.25, r\}$.} Let  $\sigma$ be sufficiently large to satisfy {$|u+p|e^{\lambda (t-\tau)} \leq r$ } for $t \geq \sigma$, $|u| \leq r$, as well as
\begin{equation}\label{K0AP}
K_0\left(  \sum_{j\in J} \int_t^{+\infty }e^{\Re \mu_j (t-s)}\|P_j(t-s)\| q(s)ds + \int_t^{+\infty } \|Q(t-s)\| q(s)ds   +  \int_\sigma^t\|Z_-(t-s)\|  q(s)ds \right) < \rho, \ t \geq \sigma.
\end{equation}
Note that since $q(s)s^{m-1}$ is
integrable on $[t_0, +\infty)$ and $\|Q(s)\| = s^{m-1}(1+o(1)), s \to +\infty,$ the second integral in (\ref{K0AP}) is converging to $0$ as $t \to +\infty$.

 {Our  goal is to prove that, for sufficiently large $\sigma$,  system (\ref{3nacAp}) has a solution $z:[\sigma+\tau, +\infty) \to \R^n$ such that $|z_j(t)| \leq 0.1 \min_{k}|p_k|$,  $t \geq \sigma+\tau$,  for each $j$.}  Clearly, this implies that the solution $y(t) = e^{\lambda t}(p +z(t)),$  {$ t > T$,}  of the original system (\ref{pel}) does not oscillate
around $0$.

To this end, for $t \geq \sigma$, consider the integral equation
\begin{equation}\label{calAp}
z(t) = \mathcal A z(t): =  -\sum_{j\in J} \int_t^{+\infty }e^{\mu_j (t-s)}P_j(t-s) M_1(s)ds - \int_t^{+\infty} Q(t-s)  M_1(s)ds
 + \int_\sigma ^t Z_-(t-s) M_1(s)ds, 
\end{equation}
where the abbreviation $M_1(s)= M_1(s, z_1(s-\tau), z_2(s),\dots,z_n(s))$ is used.
We will assume that  $z(s) \equiv 0$ on $[\sigma-\tau, \sigma)$.
Suppose that this equation has a solution $z_*: [\sigma, +\infty) \to \{u: |u|\leq r\}$ (observe that we do not claim that $z_*(\sigma) =0$, this property is not relevant for our further analysis).  Then  $z_*(t), \ t \geq \sigma+ \tau$, also satisfies system (\ref{3nacAp}). Indeed, after differentiating equation (\ref{calAp}) with respect to $t \geq \sigma+\tau$, we obtain  that
\begin{eqnarray*}
z'(t) &=& (\sum_{j\in J} P_j(0) +Q(0)+Z_-(0))M_1(t)-\sum_{j\in J} \int_t^{+\infty }(e^{\mu_j (t-s)}P_j(t-s))' M_1(s)ds\\
      &-& \int_t^{+\infty} Q'(t-s) M_1(s)ds + \int_\sigma ^t Z_-'(t-s) M_1(s)ds\\
      &=& M_1(t)-\sum_{j\in J} \int_t^{+\infty }{\mathcal L}(e^{\mu_j (t+\cdot-s)}P_j(t+\cdot-s)) M_1(s)ds\\
      &-& \int_t^{+\infty} ({\mathcal L}Q(t+\cdot-s))  M_1(s)ds + \int_\sigma ^t ({\mathcal L}Z_-(t+\cdot-s)) M_1(s)ds\\
      &=& M_1(t)+{\mathcal L}\left(-\sum_{j\in J} \int_{t+\cdot}^{+\infty }e^{\mu_j (t+\cdot-s)}P_j (t+\cdot-s)M_1(s)ds\right.\\
      &-& \left.\int_{t+\cdot}^{+\infty} Q(t+\cdot-s)  M_1(s)ds + \int_\sigma ^{t+\cdot} Z_-(t+\cdot-s) M_1(s)ds\right)\\
      &=& {\mathcal L}z_t+ M_1(t).
\end{eqnarray*}
In the fist line  of the above chain of equalities, we invoke the relation $\sum_{j\in J} P_j(0) +Q(0)+Z_-(0) = Z(0) =I$. Next, in  order to prove the equality in the fifth line, we use the standard representation (cf.  \cite[p. 193]{Hale}) of ${\mathcal L}\phi$ in the form of the Riemann-Stieltjes  integral (i.e.
${\mathcal L}\phi=\int_{-\tau}^0d\eta(s)\phi(s)$) as well as the Fubini theorem together with the property of the fundamental matrix solution that  
$Z(s)=0$ for $s \in [-\tau,0)$.

Finally, take a bounded continuous function $u(t)$, $|u(t)| \leq r,  \ t \geq \sigma,$ and set $u(t) \equiv 0$ for $t < \sigma$. {The set $\frak B$ of all such functions  equipped with
the distance $\rho(u,v) = \sup_{s \geq \sigma}|u(s)-v(s)|$ is a complete metric space. }
Then, for all $ t \geq \sigma,$ one has
$$
|\mathcal A u(t)| \leq  K_0\left(  \sum_{j\in J} \int_t^{+\infty }e^{\Re \mu_j (t-s)}\|P_j(t-s)\| q(s)ds + \int_t^{+\infty } \|Q(t-s)| q(s)ds   +  \int_\sigma^t\|Z_-(t-s)\|  q(s)ds \right) < r.
$$
{Thus the operator ${\mathcal A}: {\frak B} \to \frak{B}$ is well defined. Moreover, ${\mathcal A}$ is a contraction in view of the following relation: }
{
$$
|\mathcal A v(t)- \mathcal A u(t)| \leq  \rho \sup_{t \geq \sigma} |\left(v_1(t)- u_1(t), \dots, v_n(t)-u_n(t)\right)| \leq 0.5 \sup_{t \geq \sigma} |v(t)- u(t)|,
 \ t \geq \sigma.
$$
}Hence, the Banach contraction principle guarantees the existence of the required solution $z_*\in \frak B$. \hfill $\square$

\vspace{0.2cm}

{\bf{Acknowledgement}}.
This research project was initiated in July 2016 during the BIRS Workshop "Research in Teams", 16RiT675 (BIRS, Banff, AB, Canada).
The authors are grateful to the BIRS for providing the excellent collaborative environment for our group.
E. Braverman was partially supported by the NSERC research grant RGPIN-2015-05976.
In the follow up work on this paper, A. Ivanov was partially supported by the Alexander von Humboldt Stiftung (Germany).
S. Trofimchuk gratefully acknowledges the support of  FONDECYT (Chile), project 1150480.

The authors are grateful to the anonymous referees whose valuable comments and suggestions significantly contributed to better presentation of our results.

\end{document}